\documentclass[10pt, a4paper]{article}

\usepackage{amsthm,amssymb,mathrsfs,amsmath}
\usepackage{float}
\usepackage{color}
\usepackage[top=1in, bottom=1in, left=1in, right=1in]{geometry}
\usepackage{epsfig}

\newtheorem{theorem}{Theorem}[section]

\newtheorem{remark}{Remark}[section]
\newtheorem{definition}{Definition}[section]

\begin{document}
\title{On some computational aspects of Hermite wavelets on a class of SBVPs arising in exothermic reactions }
\author{Amit K. Verma$^a$\thanks{Email:$^a$akverma@iitp.ac.in,$^b$dikshatiwari227@gmail.com}, Diksha Tiwari$^b$
\\\small{\textit{$^{a}$Department of Mathematics, IIT Patna, Patna $801106$, Bihar, India.}}\\\small{\textit{$^b$Faculty of Mathematics, University of Vienna, Austria.}}}

\maketitle

\begin{abstract}
We propose a new class of SBVPs which deals with exothermic reactions. We also propose four computationally stable methods to solve singular nonlinear BVPs by using Hermite wavelet collocation which are coupled with Newton's quasilinearization and Newton-Raphson method. We compare the results obtained with Hermite Wavelets with Haar wavelet collocation. The efficiency of these methods are verified by applying these four methods on Lane-Emden equations. Convergence analysis is also presented.
\end{abstract}
\textbf{Keywords:} \small{MRA; Quasilinearization; Newton Raphson; Haar Wavelets; Hermite Wavelets; Nonlinear Singular Boundary value problems,}
\section{Introduction}
This paper deals with Wavelets and nonlinear singular BVPs. Nonlinear BVPs are difficult to deal and if singularity is also present it becomes even more difficult. It is not easy to capture the behavior of the solutions near the point of singularity. If we apply suitable boundary conditions which forces the existence of unique continuous solutions and hence there is possibility of finding these solutions via numerical methods. Still since the coefficient blow up when near the singularity discretizing the differential equation is a challenge. Wavelets help us to treat this complicated situation in an easy way with less number of spatial points. To address both nonlinear BVPs and Wavelets we divide the introduction in two parts.
\subsection*{Nonlinear SBVPs arising in Exothermic Reactions}\label{P2_sec1} Here we propose a new class of nonlinear SBVP. For that let us consider the mathematical equation which governs, balance between heat generated and conducted away
\begin{equation}\label{P2_CHEM1}
\lambda\nabla^{2}T=-QW,
\end{equation}
where $T$ is gas temperature, $Q$ the heat of the reaction, $\lambda$ the thermal conductivity, $W$ the reaction velocity and $\nabla^{2}$ the Laplacian operator.

Chambre \cite{CHAMBRE1952} assumed that reaction is mono-molecular and velocity follows the Arrhenius law, given as
\begin{equation}\label{P2_CHEM2}
W=A\exp\left(\frac{-E}{RT}\right),
\end{equation}
and after some approximations and symmetry assumptions Chambre \cite{CHAMBRE1952} arrived at the following equation
\begin{equation}\label{P2_CHEM3}
Ly=-\delta \exp(t)
\end{equation}
where, $$L\equiv\frac{d^2}{dt^2}+\frac{k_g}{t}\frac{d}{dt},$$ $k_g$ depends on shape and size of the vessel and $\delta$ is a parameter.

Nakamura et al. \cite{NAKAMURA1989295}, while looking for an equation which can express the temperature dependence of the rate constant proposed the following
\begin{equation}\label{P2_CHEM4}
W=A\exp\left(\frac{-E}{R~\left(T_{0}^n+T^{n}\right)^{1/n}}\right),\quad n=1,2,3,\cdots
\end{equation}
where $R$ is a gas constant, $T$ is absolute temperature, $A$, $E_{0}$ and $T_{0}$ are parameters and $n$ is an integer.

Similar to analysis of Chambre \cite{CHAMBRE1952}, we arrive at the following differential equation
\begin{equation}\label{P2_CHEM5}
Ly=-B\exp\left(\frac{-A}{\left(c^n+y^{n}\right)^{1/n}}\right).
\end{equation}
There are several other examples which led us to consider the following class of nonlinear singular boundary value problem (SBVPs)
\begin{equation}\label{P2_1}
Ly+{f(t,y(t))}=0,\quad\quad\quad 0<t\le 1,
\end{equation}
subject to following the boundary conditions
\begin{subequations}
\begin{eqnarray}
\label{P2_1a}&& Case~(i)\quad ~~y'(0)=\alpha,\quad y(1)=\beta,\\
\label{P2_1b}&& Case~(ii)\quad ~y(0)=\alpha,\quad y(1)=\beta,\\
\label{P2_1c}&& Case~(iii)\quad y'(0)=\alpha,\quad ay(1)+by'(1)=\beta,
\end{eqnarray}
\end{subequations}
where $a,b,c$, $\alpha$, $\beta$ are real constants and ${f\left(t,y(t)\right)}$ is a real valued function. Boundary conditions at the singular end $t=0$ depend on the value of $k_g$.

There is huge literature on existence of solutions of such BVPs. Please refer \cite{PANDEY2008NARWA},  \cite{PANDEY2008JMAA}, \cite{PANDEY2009NATMA} and \cite{VERMA2011NATMA} and the references there in. Several numerical methods have also been proposed for solving these type of  non-linear singular boundary value problems (see \cite{Singh2016,MSAKVRP2019,VermaKayenat2018} and its references).

In \cite{RC2013} non linear singular Lane-Emden IVPs are solved with Haar Wavelet Quasilinearization approach, and nonlinearity is easily handled with quasilinearization. In \cite{SHI2017} Hermite wavelets operational matrix method is used to solve second order nonlinear singular initial value problems. In \cite{KAZE2013} Chebyshev wavelets operation matrices are used for solving nonlinear singular boundary value problems. In \cite{Zhou2016} method based on Laguerre wavelets is used to solve nonlinear singular boundary value problems. In \cite{MOH2011} method based on Legendre wavelet is proposed to solve singular boundary value problems. All these based on wavelets show high accuracy. In \cite{akvdt2018,rsg2019,rsgjshgag2019} Haar wavelets are used to solve SBVPs efficiently for higher resolution.

In this article we solve SBVP of type \eqref{P2_1} subject to boundary conditions \eqref{P2_1a},\eqref{P2_1b},\eqref{P2_1c} with help of Hermite Wavelet Newton Approach (HeWNA), Hermite Wavelet Quasilinearization Approach (HeWQA), Haar wavelet Newton Approach (HWNA) and Haar Wavelet Quasilinearization Approach (HWQA) and compare results to show accuracy of the method. Convergence of HeWNA method is also established. Novelty of this paper is that Newton Raphson has not been coupled with Haar wavelets and Hermite has not been used to solve Nonlinear Singular BVPs till now.

Some recent works on wavelets and their applications can be find in \cite{Hariharan2019,Assari2019,SCSS2019,HKMAST2019} and the references there in.

This paper is organized in the following manner. In section \ref{mra} we discuss MRA and Hermite wavelet is defined, in sub section \ref{P2_sec3} Haar wavelet is defined, in section \ref{P2_sec4} method of solution based on HeWQA, HeWNA, HWQA  and HWNA are proposed. In section \ref{P2_sec5}, convergence analysis of HWNA method is done and in the last section some numerical examples are presented to show accuracy of the method.
\begin{section}{Hermite and Haar Wavelet}\label{mra}
This section starts with MRA and definition of Hermite Wavelets.
\subsection{Wavelets and MRA}
The theory of Wavelets are developed by mathematicians as well engineers over the years. Wavelet word find its origin the name of French geophysicist Jean Morlet (1931--2007). He used wavelet to describe certain functions. Morlet and Croatian-French physicist Alex Grossman developed theory further which is used now a days \cite[p. 222]{MCPLAW2012}. Main issue that was being addressed in the process, was to overcome drawbacks of Fourier transforms. Wavelets are multi-indexed and it contains parameters which can be used to shift or dilate/contract the functions giving us basis functions. Thus computationally they are complex but they have better control and much better results are obtained at low resolution, i.e., less number of divisions are needed that we need in finite difference and all methods based on similar concepts. Here we consider methods based on wavelet transforms.

The following properties of wavelets enable us to choose them over other methods:

\begin{itemize}
\item Orthogonality
\item Compact Support
\item Density
\item Multiresolution Analysis (MRA)
\end{itemize}
\subsubsection{MRA}
Pereyra et al. \cite{MCPLAW2012} observe that an orthogonal multiresolution analysis (MRA) is a collection of closed subspaces of $L^2(\mathbb{R})$ which are nested, having trivial intersection, they exhaust the space, the subspaces are connected to each other by scaling property and finally there is a special function, the scaling function $\varphi$, whose integer translates form an orthonormal basis for
one of the subspaces. We give formal statement of MRA as defined in \cite{MCPLAW2012}.
\begin{definition}
An MRA with scaling function $\varphi$ is a collection of closed subspaces $V_j , j = \dots,—2,—1,0,1,2,\dots$ of $L^2(\mathbb{R})$ such that
\begin{enumerate}
\item $\mathbf{V}_{j} \subset \mathbf{V}_{j+1}$
\item $\overline{\bigcup\mathbf{V}_{j}}=\mathbf{L^2(\mathbb{R})}$
\item $ \bigcap \mathbf{V}_{j}={0}$
\item The function $f(x)$ belongs to $\mathbf{V}_{j}$ if and only if the function $f(2x) \in \mathbf{V}_{j+1}$.\
\item The function $\varphi$ belongs to $\mathbf{V}_{0}$, the set $\{\varphi(x-k), k \in \mathbf{\mathbb{Z}}\}$ is orthonormal basis for $\mathbf{V}_{0}$.
 \end{enumerate}
\end{definition}
The sequence of wavelet subspaces $W_j$ of $L^2(\mathbb{R})$, are such that $V_j\perp Wj$ , for all $j$ and $V_{j+1} = V_j\oplus W_j$. Closure of $\oplus_{j\in \mathbb{Z}} W_j$ is dense in $L^2(\mathbb{R})$ with respect to $L^2$ norm.

Now we state Mallat's theorem \cite{MALLAT1989} which guarantees that in presence of an orthogonal MRA, an orthonormal basis for $L^2(\mathbb{R})$ exists. These basis functions are fundamental functions in the theory of wavelets which helps us to develop advanced computational techniques.
\begin{theorem} {$\mathrm{(Mallat's~Theorem)}$}. Given an orthogonal MRA
with scaling function $\varphi$, there is a wavelet $\psi\in L^2(\mathbb{R})$ such that for
each $j\in \mathbb{Z}$, the family $\{\psi_{j,k}\}_{k\in\mathbb{Z}}$ is an orthonormal basis for $W_j$ .
Hence the family $\{\psi_{j,k}\}_{k\in\mathbb{Z}}$ is an orthonormal basis for $L^2(\mathbb{R})$. \label{Mallat}
\end{theorem}
\subsection{Hermite Wavelet (\cite{Saeed2014})}\label{P2_sec2}
Hermite Polynomials are defined on the interval $(-\infty,\infty)$ and can be defined with help of the recurrence formula:
\begin{eqnarray*}
&&H_{0}(t)=1\\
&&H_{1}(t)=2t\\
&&H_{m+1}(t)= 2tH_{m}(t)-2mH_{m-1}(t),~m=1,2,3,\cdots.
\end{eqnarray*}
Completeness and orthogonality (with respect to weight function $e^{-t^2}$) of Hermite polynomials enable us to treat them as wavelet (Theorem \ref{Mallat}).

Hermite wavelet on the interval $[0,1]$ are defined as
\begin{equation}
\label{P2_Def1}
\psi_{n,m}(t)=
2^{k/2}\frac{1}{\sqrt{n!2^{n}\sqrt{\pi}}}H_{m}(2^kt-\hat{n})\chi_{\left[\frac{\hat{n}-1}{2^k},\frac{\hat{n}+1}{2^k}\right)}
\end{equation}
where $k=1,2,\hdots$ is level of resolution, $n=1,2,\hdots, 2^{k-1},\hspace{0.25cm} \hat{n}=2n-1$ is translation parameter, $m=1,2,\hdots,M-1$ is order of Hermite polynomial.
\end{section}
\subsubsection{Approximation of Function with Hermite Wavelet}\label{P2_subsec2}
A function $f(t)$ defined on $L^2[0,1]$ can be approximated with Hermite wavelet in the following manner
\begin{equation}\label{P2_2}
f(t)= \sum^{\infty}_{n=1}\sum^{\infty}_{m=0}c_{nm}\psi_{nm}(t).
\end{equation}
For computation purpose we truncate \eqref{P2_2} and define,
\begin{equation}\label{P2_3}
f(t)\simeq \sum^{2^{k}-1}_{n=1}\sum^{M-1}_{m=0}c_{nm}\psi_{nm}(t)=c^{T}\psi(t)
\end{equation}
where  $\psi(t)$ is $2^{k-1}M\times 1$ matrix given as:
\begin{equation*}
\psi(t)= \left[\psi_{1,0}(t),\dots,\psi_{1,M-1}(t),\psi_{2,0}(t),\dots,\psi_{2,M-1}(t),\dots,\psi_{2^{k-1},0}(t),\dots,\psi_{2^{k-1},M-1}(t)\right]^T
\end{equation*}
$c$ is $2^{k-1}M\times 1$ matrix. Entries of $c$ can be computed as
\begin{equation}\label{P2_4}
c_{ij}=\int^{1}_{0}f(t)\psi_{ij}(t)dt
\end{equation}
with $i=1,2,\dots,2^k-1$ and $j=0,1,\dots,M-1$. Here $M$ is degree of Hermite polynomial.
\subsubsection{Integration of Hermite Wavelet}
As suggested in \cite{Gupta2015}, $\nu$-th order integration of $\psi(t)$ can also be approximated as
\begin{multline*}\label{P2_5}
\int^{t}_{0}\int^{t}_{0}\dots\int^{t}_{0}\psi(\tau)d\tau\simeq\\
\left[J^\nu\psi_{1,0}(t),\dots,J^\nu\psi_{1,M-1}(t),J^\nu\psi_{2,0}(t),\dots,J^\nu\psi_{2,M-1}(t),\dots,J^\nu\psi_{2^{k-1},0}(t),\dots,J^\nu\psi_{2^{k-1},M-1}(t)\right]^T
\end{multline*}
where
\begin{equation}
J^{\nu}\psi_{n,m}(t)= 2^{k/2}\frac{1}{\sqrt{n!2^{n}\sqrt{\pi}}}J^{\nu}H_{m}(2^kt-\hat{n})\chi_{\left[\frac{\hat{n}-1}{2^k},\frac{\hat{n}+1}{2^k}\right)},
\end{equation}
where $k=1,2,\hdots$ is level of resolution, $n=1,2,\hdots, 2^{k-1},\hspace{0.25cm} \hat{n}=2n-1$ is translation parameter, $m=1,2,\hdots,M-1$ is order of Hermite polynomial.
\begin{remark} Integral operator $J^{\nu}(\nu>0)$ of a function $f(t)$ is defined as
\begin{equation*}\label{P2_6}
J^{\nu}f(t)= \frac{1}{\nu!}\int_{0}^{t}(t-s)^{\nu-1}f(s)ds.
\end{equation*}
\end{remark}
\subsubsection{Hermite Wavelet Collocation Method}\label{P2_subsec3}
To apply Hermite wavelet on the ordinary differential equations, we need its discretized form. We use collocation method for discretization where mesh points are given by
\begin{eqnarray}
\label{P2_19}\bar x_{l} = l \Delta x,\quad\quad\quad l = 0,1,\cdots,M-1.
\end{eqnarray}
For the collocation points we define
\begin{eqnarray}
\label{P2_20}x_{l} = 0.5(\bar x_{l-1} + \bar x_{l}), \quad\quad\quad l = 0,1,\cdots,M-1.
\end{eqnarray}
For $k=1$, equation \eqref{P2_3} takes the form
\begin{equation}\label{P2_21}
f(t)\simeq \sum^{M-1}_{m=0}c_{1m}\psi_{1m}(t).
\end{equation}
We replace $t$ by $x_{l}$ in above equation and arrive at system of equations which can easily be solved to get the solution of the nonlinear SBVP.
\subsection{Haar Wavelet}\label{P2_sec3}
Let us assume that $x$ belongs to any interval $[P, Q]$, where $P$ and $Q$ are constant end points. Let us define $M = 2^{J}$, where $J$ is the maximal level of resolution. Devide $[P, Q]$ into $2M$ subintervals of equal length $\Delta x = (Q- P)/(2M)$. The wavelet number $i$ is defined as $i = m + k + 1$, where $j = 0, 1,\cdots,J$ and $k = 0, 1,\cdots,m - 1$ (here $m = 2j$).
The $i^{th}$ Haar wavelet is explained as
\begin{equation}
\label{P2_7}h_{i}(x)= \chi_{\left[{\eta }_{1}\left(i\right),{\eta }_{2}\left(i\right)\right)}-\chi_{\left[{\eta }_{2}\left(i\right),{\eta }_{3}\left(i\right)\right)}
\end{equation}
where
\begin{eqnarray}
\label{P2_8}\eta_{1} (i) = P + 2k\mu\Delta x,\quad
\eta_{2}(i) = P + (2k + 1)\mu\Delta x,\quad
\eta_{3}(i) = P + 2(k + 1)\mu\Delta x,\quad
\mu= M/m.
\end{eqnarray}
Above equations are valid for $i > 2$. For $i = 1$ case we have, $h_{i}(x) = \chi_{[P,Q]}$.

For $i = 2$ we have
\begin{eqnarray}
\label{P2_9}\eta_{1}(2) = P,\quad
\eta_{2}(2) = 0.5(2P + Q),\quad
\eta_{3}(2) = Q.
\end{eqnarray}
The thickness of the $i^{th}$ wavelet is
\begin{eqnarray}
\eta_{3}(i)- \eta_{1}(i) = 2\mu\Delta x = (Q- P)m^{-1} = (Q- P)2^{-j}.
\end{eqnarray}
If $J$ is fixed then by \eqref{P2_7}
\begin{eqnarray}
\label{P2_10}\int_{P}^{Q} h_{i}(x)h_{l}(x)dx=
\begin{cases}
(Q-P)2^{-j}, & l = i, \\
0, & l\neq i.
\end{cases}
\end{eqnarray}
The integrals $p_{\upsilon,i}(x)$ are defined as
\begin{eqnarray}
\label{P2_11} p_{\upsilon,i}(x)=\int_{P}^{x}\int_{P}^{x}\cdots\int_{P}^{x}h_{i}(t)dt^{v}=\frac{1}{(v-1)!}\int_{P}^{x}(x- t)^{(\upsilon-1)}h_{i}(t)dt,
\end{eqnarray}
where $\upsilon = 1,2,\cdots,n,\quad i = 1,2,\cdots,2M.$

Putting all values in the integral we get
\begin{multline}\label{P2_12}
p_{\upsilon , i}(x)=
\frac{1}{\upsilon!}[x-\eta_{1}(i)]^\upsilon\chi_{\left[{\eta}_{1}\left(i\right),{\eta}_{2}\left(i\right)\right)}+ \frac{1}{\upsilon!}\{[x-\eta_{1}(i)]^\upsilon-2[x-\eta_{2}(i)]^\upsilon\} \chi_{\left[{\eta }_{2}\left(i\right),{\eta }_{3}\left(i\right)\right]}\\
+\frac{1}{\upsilon!}\{[x-\eta_{1}(i)]^\upsilon-2[x-\eta_{2}(i)]^\upsilon+[x-\eta_{3}(i)]^\upsilon\}\chi_{(\eta_{3}(i),\infty)}
\end{multline}
for $i>1$ and for $i=1$ we have $\eta_{1}=P, \eta_{2} = \eta_{3}=Q$ and
\begin{equation}
\label{P2_13}p_{\upsilon,1}(x)=\frac{1}{\upsilon!}(x-P)^\upsilon.
\end{equation}

\subsubsection{Haar Wavelet Collocation Method}\label{P2_subsec3b}
Similar to case of previous section here again we define collocation points as follows
\begin{eqnarray}
&&\label{P2_14}\bar x_{t} = P + t \Delta x,\quad\quad\quad~ t = 0,1,\cdots,2M,\\
&&\label{P2_15}x_{t} = 0.5(\bar x_{t-1} + \bar x_{t}), \quad t = 0,1,\cdots,2M,
\end{eqnarray}
and replace $x \to x_{t}$ in \eqref{P2_7},\eqref{P2_8},\eqref{P2_9}. We define the Haar matrices $H, P_{1},P_{2},\cdots, P_\upsilon$ which are $2M \times 2M$ matrices. Entries of matrices are given by $H(i,t)=h_{i}(x_{t})$, $P_{v}(i,t)=P_{\upsilon, i}(x_{t})$, $\upsilon = 1,2,\cdots$.
Consider $P=0, Q=1,J=1$. Then $2M=4$, so $H,P_{1},P_{2}$ are defined as
\begin{eqnarray}
\label{P2_16}H=\left[
\begin{array}{cccc}
1& 1& 1& 1\\
1& 1& -1& -1\\
1& -1& 0& 0\\
0& 0& 1& -1
\end{array}\right],
\quad P_{1}=\frac{1}{8}
\left[
\begin{array}{cccc}
1& 3& 5& 7\\
1& 3& 3& 1\\
1& 1& 0& 0\\
0& 0& 1& 1
\end{array}\right],
\quad P_{2}=\frac{1}{128}\left[
\begin{array}{cccc}
1& 9& 25& 49\\
1& 9& 23& 31\\
1& 7& 8& 8\\
0& 0& 1& 7
\end{array}\right].
\end{eqnarray}
\subsubsection{Approximation of Function with Haar Wavelet}\label{P2_subsec3c}
A function $f(t)$ defined on $L^2[0,1]$ can be approximated by Haar wavelet basis in the following manner
\begin{equation}\label{P2_17}
f(t)= \sum^{\infty}_{i=0}a_{i}h_{i}(t).
\end{equation}
For computation purpose we truncate \eqref{P2_17} and define
\begin{equation}\label{P2_18}
f(t)\simeq \sum^{2M}_{i=0}a_{i}h_{i}(t),
\end{equation}
where $M$ is level of resolution.
\section{Numerical Methods Based on Hermite and Haar Wavelets}\label{P2_sec4}
In this section, solution methods based on Hermite wavelet and Haar wavelet are presented.
\subsection{Hermite Wavelet Quasilinearization Approach (HeWQA)}\label{P2_subsec4a}
In HeWQA we use quasilinearization to linearize SBVP then method of collocation for discretization and use Hermite wavelets for computation of solutions to nonlinear SBVP. We consider differential equation \eqref{P2_1} with boundary conditions \eqref{P2_1a}, \eqref{P2_1b} and \eqref{P2_1c}. Quasilinearizing this equation, we get the form
\begin{subequations}
\begin{eqnarray}
\label{P2_22a}Ly_{r+1}=-f(t,y_{r}(t))+\sum_{s=0}^{1}(y_{r+1}^{s}-y_{r}^{s})(-f_{y^{s}}(t,y_{r}(t)),
\end{eqnarray}
subject to linearized boundary conditions,
\begin{eqnarray}
\label{P2_22b}&&y'_{r+1}(0)=y_{r}(0),\quad{y_{r+1}(1)}={y_{r}(1)},\\
\label{P2_22c}&&y_{r+1}(0)=y_{r}(0),\quad{y_{r+1}(1)}={y_{r}(1)},\\
\label{P2_22d}&&y'_{r+1}(0)=y_{r}(0),\quad a{y_{r+1}(1)}+by'_{r+1}(1)=a{y_{r}(1)}+by'_{r}(1).
\end{eqnarray}
Here $s=0,1$, $f_{y^{s}}=\partial{f}/\partial{y^{s}}$ and $y_{r}^{0}(t)=y_{r}(t)$.

Thus we arrive at linearized form of given differential equation. Now we use Hermite wavelet method similar to described in \cite{CH1997}. Let us assume
\begin{eqnarray}
\label{P2_22e}y''_{r+1}(t)=\sum^{M-1}_{m=0}c_{1m}\psi_{1m}(t).
\end{eqnarray}
Then integrating twice we get the following two equations:
\begin{eqnarray}
&&\label{P2_22f}y'_{r+1}(t)=\sum^{M-1}_{m=0}c_{1m}J\psi_{1m}(t)+y'_{r+1}(0),\\
&&\label{P2_22g}y_{r+1}(t)=\sum^{M-1}_{m=0}c_{1m}J^2\psi_{1m}(t)+ty'_{r+1}(0)+y_{r+1}(0).
\end{eqnarray}
Here $J^{\nu}(\nu>0)$ is the integral operator defined previously.
\end{subequations}
\subsubsection{Treatement of the Boundary Value Problem}\label{P2_subsubsec4a}
Based on boundary conditions we will consider different cases and follow procedure similar to described in \cite{LU2008}.
\textbf{Case (i)}:\quad In equation \eqref{P2_1a} we have 	
$y'(0)=\alpha,y(1)=\beta$. So by linearization we have $y'_{r+1}(0)=\alpha,y_{r+1}(1)=\beta$. Now put $t=1$ in \eqref{P2_22g} we get
\begin{equation*}\label{P2_23}
y_{r+1}(1)=\sum^{M-1}_{m=0}c_{1m}J^2\psi_{1m}(1)+y'_{r+1}(0)+y_{r+1}(0),
\end{equation*}
so
\begin{equation}
\label{P2_24}y_{r+1}(0)=y_{r+1}(1)-\sum^{M-1}_{m=0}c_{1m}J^2\psi_{1m}(1)-y'_{r+1}(0).
\end{equation}
By using equation \eqref{P2_24} in \eqref{P2_22g} and simplifying we get
\begin{equation}\label{P2_27}
y_{r+1}(t)=(t-1)y'_{r+1}(0)+y_{r+1}(1)+\sum^{M-1}_{m=0}c_{1m}(J^2\psi_{1m}(t)-J^2\psi_{1m}(1)).
\end{equation}
Now putting values of $y'_{r+1}(0)$ and $y_{r+1}(1)$ in \eqref{P2_22f} and \eqref{P2_27} we get
\begin{eqnarray}
&&y'_{r+1}(t)=\alpha +\sum^{M-1}_{m=0}c_{1m}J\psi_{1m}(t),\\
&&\label{P2_28}
y_{r+1}(t)=(t-1)\alpha + \beta +\sum^{M-1}_{m=0}c_{1m}(J^2\psi_{1m}(t)-J^2\psi_{1m}(1)).
\end{eqnarray}
\textbf{Case (ii)}:\quad In equation \eqref{P2_1b} we have 	
$y(0)=\alpha,y(1)=\beta$. So by linearization we have $y_{r+1}(0)=\alpha,y_{r+1}(1)=\beta$. Now put $t=1$ in equation \eqref{P2_22g} we get
\begin{eqnarray}\label{P2_29}
y_{r+1}(1)=\sum^{M-1}_{m=0}c_{1m}J^2\psi_{1m}(1)+y'_{r+1}(0)+y_{r+1}(0),
\end{eqnarray}
so
\begin{eqnarray}\label{P2_30}
y'_{r+1}(0)=y_{r+1}(1)-\sum^{M-1}_{m=0}c_{1m}J^2\psi_{1m}(1)-y_{r+1}(0).
\end{eqnarray}
By putting these values in equation \eqref{P2_22f} and \eqref{P2_22g} we get
\begin{eqnarray*}
y'_{r+1}(t)=y_{r+1}(1)-y_{r+1}(0)+\sum^{M-1}_{m=0}c_{1m}(J\psi_{1m}(t)-J^2\psi_{1m}(1)),
\end{eqnarray*}
and
\begin{eqnarray*}\label{P2_31}
y_{r+1}(t)=(1-t)y_{r+1}(0)+ty_{r+1}(1)+\sum^{M-1}_{m=0}c_{1m}(J^2\psi_{1m}(t)-J^2\psi_{1m}(1)).
\end{eqnarray*}
Now we put values of $y_{r+1}(0)$ and $y_{r+1}(1)$ we get
\begin{eqnarray}\label{P2_32}
&&y'_{r+1}(t)=(\beta-\alpha)+\sum^{M-1}_{m=0}c_{1m}(J\psi_{1m}(t)-J^2\psi_{1m}(1)),\\
&&\label{P2_33}
y_{r+1}(t)=(1-t)\alpha + t\beta +\sum^{M-1}_{m=0}c_{1m}J^2\psi_{1m}(1)-y_{r+1}(0)).
\end{eqnarray}
\textbf{Case (iii)}:\quad In equation \eqref{P2_1c} we have 	
$y'(0)=\alpha, ay(1)+by'(1)=\beta$. So by linearization we have $y'_{r+1}(0)=\alpha,ay_{r+1}(1)+by'_{r+1}(1)=\beta$. Now put $t=1$ in equation \eqref{P2_22f} and \eqref{P2_22g} we get
\begin{eqnarray}\label{P2_34}
&&y'_{r+1}(1)=\sum^{M-1}_{m=0}c_{1m}J\psi_{1m}(1)+y'_{r+1}(0),\\
\label{P2_35}&&y_{r+1}(1)=\sum^{M-1}_{m=0}c_{1m}J^2\psi_{1m}(1)+y'_{r+1}(0)+y_{r+1}(0).
\end{eqnarray}
Putting these values in $ay_{r+1}(1)+by'_{r+1}(1)=\beta$ and solving for $y_{r+1}(0)$ we have
\begin{eqnarray*}\label{P2_37}
y_{r+1}(0)=\frac{1}{a}\left(\beta - ay'_{r+1}(0)-a\sum^{M-1}_{m=0}c_{1m}J^2\psi_{1m}(1)-b\left(\sum^{M-1}_{m=0}c_{1m}J\psi_{1m}(1)+y'_{r+1}(0)\right)\right).
\end{eqnarray*}
Hence from equation \eqref{P2_22g} we have
\begin{multline}\label{P2_38}
y_{r+1}(t)=\sum^{M-1}_{m=0}c_{1m}J^2\psi_{1m}(t)+ty'_{r+1}(0)\\
+\frac{1}{a}\left(\beta -ay'_{r+1}(0)-a\sum^{M-1}_{m=0}c_{1m}J^2\psi_{1m}(1)-b\left(\sum^{M-1}_{m=0}c_{1m}J\psi_{1m}(1)+y'_{r+1}(0)\right)\right).
\end{multline}
Now we put values of $y_{r+1}(0)$ and $y_{r+1}(1)$ in equation \eqref{P2_22f} and \eqref{P2_38} we get
\begin{eqnarray}
&&y'_{r+1}(t)=\alpha +\sum^{M-1}_{m=0}c_{1m}J\psi_{1m}(t),\\
&&\label{P2_39}y_{r+1}(t)=\frac{\beta}{a}+\left(t-1-\frac{b}{a}\right)\alpha+\sum^{M-1}_{m=0}c_{1m}\left(J^2\psi_{1m}(t)-J^2\psi_{1m}(1)-\frac{b}{a}J\psi_{1m}(1)\right).
\end{eqnarray}
Finally we put values of $y''_{r+1}$, $y'_{r+1}$ and $y_{r+1}$ for all these cases in the linearized differential equation \eqref{P2_22a}. Now we discritize the final equation with collocation method and then solve the resulting system assuming initial guess $y_{0}(t)$. We get required value of solution $y(t)$ of the nonlinear SBVPs at different collocation points.
\subsection{Hermite Wavelet Newton Approach (HeWNA)}\label{P2_subsec4b}
In this approach we use the method of collocation for discretization and then Hermite wavelet for approximation of the solutions. Finally Newton Raphson method is used to solve the resulting nonlinear system of equations.

We consider differential equation \eqref{P2_1} with boundary conditions \eqref{P2_1a}, \eqref{P2_1b} and \eqref{P2_1c}. Now we assume
\begin{eqnarray}
\label{P2_70}y''(t)=\sum^{M-1}_{m=0}c_{1m}\psi_{1m}(t).
\end{eqnarray}
Integrating twice we get the following two equations:
\begin{eqnarray}
&&\label{P2_71}y'(t)=\sum^{M-1}_{m=0}c_{1m}J\psi_{1m}(t)+y'(0),\\
&&\label{P2_72}y(t)=\sum^{M-1}_{m=0}c_{1m}J^2\psi_{1m}(t)+ty'(0)+y(0).
\end{eqnarray}
\subsubsection{Treatment of the Boundary Value Problem}\label{P2_subsubsec4b}
Based on boundary conditions we divide it in different cases.\\
\textbf{Case (i)}: In equation \eqref{P2_1a} we have 	
$y'(0)=\alpha, y(1)=\beta$. Now put $t=1$ in \eqref{P2_72} we get
\begin{eqnarray}\label{P2_40}
y(1)=\sum^{M-1}_{m=0}c_{1m}J^2\psi_{1m}(1)+y'(0)+y(0),
\end{eqnarray}
so
\begin{eqnarray}\label{P2_41}
y(0)=y(1)-\sum^{M-1}_{m=0}c_{1m}J^2\psi_{1m}(1)-y'(0).
\end{eqnarray}
By using equation \eqref{P2_41} in \eqref{P2_72} we get
\begin{eqnarray*}
\label{P2_42} &&y(t)=\sum^{M-1}_{m=0}c_{1m}J^2\psi_{1m}(t)+ty'(0)+y(1)-\sum^{M-1}_{m=0}c_{1m}J^2\psi_{1m}(1)-y'(0),\\
\label{P2_43}
&&y(t)=\sum^{M-1}_{m=0}c_{1m}J^2\psi_{1m}(t)+(t-1)y'(0)+y(1)-\sum^{M-1}_{m=0}c_{1m}J^2\psi_{1m}(1).
\end{eqnarray*}
Hence
\begin{eqnarray}\label{P2_44}
y(t)=(t-1)y'(0)+y(1)+\sum^{M-1}_{m=0}c_{1m}(J^2\psi_{1m}(t)-J^2\psi_{1m}(1)).
\end{eqnarray}
Now putting values of $y'(0)$ and $y(1)$ in \eqref{P2_71} and \eqref{P2_44} we get
\begin{eqnarray}
&&y(t)=\sum^{M-1}_{m=0}c_{1m}J^2\psi_{1m}(t)+ty'(0)+y(0),\\
&&\label{P2_45} y(t)=(t-1)\alpha + \beta +\sum^{M-1}_{m=0}c_{1m}(J^2\psi_{1m}(t)-J^2\psi_{1m}(1)).
\end{eqnarray}
\textbf{Case (ii)}: In equation \eqref{P2_1b} we have 	
$y(0)=\alpha,y(1)=\beta$. Now put $t=1$ in equation \eqref{P2_72} we get
\begin{eqnarray}\label{P2_46}
y(1)=\sum^{M-1}_{m=0}c_{1m}J^2\psi_{1m}(1)+y'(0)+y(0),
\end{eqnarray}
so
\begin{eqnarray*}\label{P2_47}
y'(0)=y(1)-\sum^{M-1}_{m=0}c_{1m}J^2\psi_{1m}(1)-y(0).
\end{eqnarray*}
Now using these values of $y'(0)$ and $y(1)$ in \eqref{P2_71} and \eqref{P2_72} and solving we get
\begin{eqnarray}
&&y'(t)=(\beta-\alpha)+\sum^{M-1}_{m=0}c_{1m}(J\psi_{1m}(t)-J^2\psi_{1m}(1)),\\
&&\label{P2_48}
y(t)=(1-t)\alpha + t\beta +\sum^{M-1}_{m=0}c_{1m}(J^2\psi_{1m}(t)-J^2\psi_{1m}(1)).
\end{eqnarray}
\textbf{Case (iii)}: In equation \eqref{P2_1c} we have 	
$y'(0)=\alpha, ay(1)+by'(1)=\beta$. Now put $t=1$ in equation \eqref{P2_71} and \eqref{P2_72} we get
\begin{eqnarray}\label{P2_49}
&&y'(1)=\sum^{M-1}_{m=0}c_{1m}J\psi_{1m}(1)+y'(0),\\
&&\label{P2_50}y(1)=\sum^{M-1}_{m=0}c_{1m}J^2\psi_{1m}(1)+y'(0)+y(0).
\end{eqnarray}
Putting these values in $ ay(1)+by'(1)=\beta$ and solving we will get value of $y(0)$, now put $y(0)$ and $y'(0)$ in \eqref{P2_72} we have
\begin{eqnarray}\label{P2_51}
y(t)=\sum^{M-1}_{m=0}c_{1m}J^2\psi_{1m}(t)+t y'(0)+\frac{1}{a}\left(\beta - ay'(0)-a\sum^{M-1}_{m=0}c_{1m}J^2\psi_{1m}(1)-b\left(\sum^{M-1}_{m=0}c_{1m}J\psi_{1m}(1)+y'(0)\right)\right).
\end{eqnarray}
Now by putting $y'(0)=\alpha$ in \eqref{P2_71} and \eqref{P2_51}, we have
\begin{eqnarray}
&&y'(t)=\alpha +\sum^{M-1}_{m=0}c_{1m}J\psi_{1m}(t),\\
&&\label{P2_52}
y(t)=\frac{\beta}{a}+\left(t-1-\frac{b}{a}\right)\alpha+\sum^{M-1}_{m=0}c_{1m}\left(J^2\psi_{1m}(t)-J^2\psi_{1m}(1)-\frac{b}{a}J\psi_{1m}(1)\right).
\end{eqnarray}
Now we put values of $y(t)$, $y'(t)$ and $y''(t)$ in \eqref{P2_1}. We discretize final equation with collocation method and solve the resulting nonlinear system with Newton Raphson method for $c_{1m}, m=0,1,\dots,M-1$. By substituting the values of $c_{1m}, m=0,1,\dots,M-1$, we get value of the solution of $y(t)$ of nonlinear SBVPs at different collocation points.
\subsection{Haar Wavelet Quasilinearization Approach (HWQA)}\label{P2_subsec4c}
As explained for HeWQA, we will follow same procedure in HWQA method. Here we are using Haar Wavelet in place of Hermite Wavelet. We consider differential equation \eqref{P2_1} with boundary conditions \eqref{P2_1a}, \eqref{P2_1b} or \eqref{P2_1c}. Applying method of quasilinearization, as we did in HeWQA,  we have equation \eqref{P2_22a} with linearized boundary conditions \eqref{P2_22b},\eqref{P2_22c} and \eqref{P2_22d}.

Let us assume
\begin{eqnarray}\label{P2_53e}
y''_{r+1}(t)=\sum_{i=0}^{2M}a_{i}h_{i}(t),
\end{eqnarray}
where $a_{i}$ are the wavelet coefficients. Then integrating twice we get following two equations:
\begin{eqnarray}\label{P2_53f}
&&y'_{r+1}(t)=\sum_{i=0}^{2M}a_{i}p_{1,i}(t)+y'_{r+1}(0),\\
&&\label{P2_53g}y_{r+1}(t)=\sum_{i=0}^{2M}a_{i}p_{2,i}(t)+ty'_{r+1}(0)+y_{r+1}(0).
\end{eqnarray}
\subsubsection{Treatement of the Boundary Value Problem}\label{P2_subsubsec4c}
Expressions for different boundary conditions in HWQA method are given below.\\
\noindent \textbf{Case (i)}: In equation \eqref{P2_1a} we have 	
$y'(0)=\alpha, y(1)=\beta$. Following same procedure as HeWQA
we have
\begin{eqnarray}
&&y'_{r+1}(t)=\alpha+\sum_{i=0}^{2M}a_{i}p_{1,i}(t),\\
&&\label{P2_54}y_{r+1}(t)=(t-1)\alpha + \beta +\sum_{i=0}^{2M}a_{i}(p_{2,i}(t)-p_{2,i}(1)).
\end{eqnarray}
\noindent \textbf{Case (ii)}: In equation \eqref{P2_1b} we have $y(0)=\alpha,y(1)=\beta$. So we have
\begin{eqnarray}
&&y'_{r+1}(t)=(\beta-\alpha)+\sum^{M-1}_{m=0}c_{1m}(p_{1,i}(t)-p_{2,i}(1)),\\
&&\label{P2_55}y_{r+1}(t)=(1-t)\alpha + t\beta +\sum_{i=0}^{2M}a_{i}(p_{2,i}(t)-tp_{2,i}(1)).
\end{eqnarray}
\textbf{Case (iii)}: In equation \eqref{P2_1c} we have $y'(0)=\alpha, ay(1)+by'(1)=\beta$. We finally have
\begin{eqnarray}
&&y'_{r+1}(t)=\alpha+\sum_{i=0}^{2M}a_{i}p_{1,i}(t),\\
&&\label{P2_57} y_{r+1}(t)=\frac{\beta}{a}+\left(t-1-\frac{b}{a}\right)\alpha+\sum_{i=0}^{2M}a_{i}\left(p_{2,i}(t)-p_{2,i}(1)-\frac{b}{a}p_{1,i}(1)\right).
\end{eqnarray}
\subsection{Haar Wavelet Newton Approach(HWNA)}\label{P2_subsec4d}
Here we use same procedure as that of HeWNA.
\subsubsection{Treatment of the Boundary Value Problem}\label{P2_subsubsec4d}
Expressions for different boundary conditions in HWNA method are given below.\\
\textbf{Case (i)}: In equation \eqref{P2_1a} we have 	
$y'(0)=\alpha,y(1)=\beta$. Following same procedure, final expression will take the form
\begin{eqnarray}
&&y'(t)=\alpha+\sum_{i=0}^{2M}a_{i}p_{1,i}(t),\\
&&\label{P2_58}y(t)=(t-1)\alpha + \beta +\sum^{2M}_{m=0}a_{i}(p_{2,i}(t)-p_{2,i}(1)).
\end{eqnarray}
\noindent\textbf{Case (ii)}: In equation \eqref{P2_1b} we have 	
$y(0)=\alpha,y(1)=\beta$. So by linearization we have $y(0)=\alpha,~y(1)=\beta$. Final expression is of the form
\begin{eqnarray}
&&y'(t)=(\beta-\alpha)+\sum^{M-1}_{m=0}c_{1m}(p_{1,i}(t)-p_{2,i}(1)),\\
&&\label{P2_59}
y(t)=(1-t)\alpha + t\beta +\sum^{2M}_{m=0}c_{1m}(p_{2,i}(t)-p_{2,i}(1)).
\end{eqnarray}
\textbf{Case (iii)}: In equation \eqref{P2_1c} we have 	
$y'(0)=\alpha,~ ay(1)+by'(1)=\beta$, so we have the following expression:
\begin{eqnarray}
&&y'(t)=\alpha+\sum_{i=0}^{2M}a_{i}p_{1,i}(t),\\
&&\label{P2_60} y(t)=\frac{\beta}{a}+\left(t-1-\frac{b}{a}\right)\beta+\sum^{2M}_{m=0}c_{1m}\left(p_{2,i}(t)-p_{2,i}(1)-\frac{b}{a}p_{1,i}(1)\right).
\end{eqnarray}
\section{Convergence}\label{P2_sec5}
Let us consider $2^{\mathrm{nd}}$ order ordinary differential equation in general form
\begin{equation*}
G(t,u,u',u'')=0.
\end{equation*}
We consider the HeWNA method. Let
\begin{equation}\label{P2_con1}
f(t)=u''(t)= \sum^{\infty}_{n=1}\sum^{\infty}_{m=0}c_{nm}\psi_{nm}(t).
\end{equation}
Integrating above equation two times we have
\begin{equation}\label{P2_con2}
u(t)= \sum^{\infty}_{n=1}\sum^{\infty}_{m=0}c_{nm}J^2\psi_{nm}(t)+B_{T}(t),
\end{equation}
where $B_{T}(t)$ stands for boundary term.
\begin{theorem}
Let us assume that, $f(t)=\frac{d^2u}{dt^2}\in L^2(\mathbb{R})$ is a continuous function defined on [0,1]. Let us consider $f(t)$ is bounded, i.e.,
\begin{equation}\label{P2_con3}
\forall  t \in [0,1] \quad \quad \exists \quad \eta :\left|\frac{d^2u}{dt^2}\right| \leq \eta.
\end{equation}
Then method based on Hermite Wavelet Newton Approach (HeWNA) converges.
\end{theorem}
\begin{proof}
In \eqref{P2_con2} by truncating expansion we have,
\begin{eqnarray}\label{P2_con4}
u^{k,M}(t)= \sum^{2^k-1}_{n=1}\sum^{M-1}_{m=0}c_{nm}J^2\psi_{nm}(t)+B_{T}(t)
\end{eqnarray}
So error $E_{k,M}$  can be expressed as
\begin{eqnarray}\label{P2_con5}
\|E_{k,M}\|_{2}=\|u(t)-u^{k,M}(t)\|_{2}=\left\|\sum^{\infty}_{n=2^k}\sum^{\infty}_{m=M}c_{nm}J^2\psi_{nm}(t)\right\|_{2}.
\end{eqnarray}
Expanding $L^{2}$ norm, we have
\begin{eqnarray}\label{P2_con6}
&&\|E_{k,M}\|_2^2=\int_{0}^{1}\left(\sum^{\infty}_{n=2^k}\sum^{\infty}_{m=M}c_{nm}J^2\psi_{nm}(t)\right)^2dt,\\
\label{P2_con8}
&&\|E_{k,M}\|_2^2=\sum^{\infty}_{n=2^k}\sum^{\infty}_{m=M}\sum^{\infty}_{s=2^k}\sum^{\infty}_{r=M}\int_{0}^{1}c_{nm}c_{sr} J^2\psi_{nm}(t)J^2\psi_{sr}(t)dt,\\
&&\label{P2_con9}
\|E_{k,M}\|_2^2\leq \sum^{\infty}_{n=2^k}\sum^{\infty}_{m=M}\sum^{\infty}_{s=2^k}\sum^{\infty}_{r=M}\int_{0}^{1}|c_{nm}|~|c_{sr}|~ |J^2\psi_{nm}(t)|~|J^2\psi_{sr}(t)|dt.
\end{eqnarray}
Now
\begin{eqnarray*}\label{P2_con7}
&&|J^{2}\psi_{nm}(t)|\leq \int_{0}^{t}\int_{0}^{t}|\psi_{nm}(t)|dtdt,\\
\label{P2_con10}
&&\leq \int_{0}^{t}\int_{0}^{1}|\psi_{nm}(t)|dtdt,
\end{eqnarray*}
since $t  \in [0,1]$. Now by \eqref{P2_Def1}, we have
\begin{eqnarray*}\label{P2_con11}
|J^{2}\psi_{nm}(t)|\leq  2^{k/2}\frac{1}{\sqrt{n!2^{n}\sqrt{\pi}}}\int_{0}^{t}\int_{\frac{\hat{n}-1}{2^k}}^{\frac{\hat{n}+1}{2^k}}|H_{m}(2^{k}t-\hat{n})|dtdt.
\end{eqnarray*}
By changing variable $2^kt-\hat{n}=y$, we get
\begin{eqnarray*}
|J^{2}\psi_{nm}(t)|&&\leq 2^{-k/2}\frac{1}{\sqrt{n!2^{n}\sqrt{\pi}}}\int_{0}^{t}\int_{-1}^{1}|H_{m}(y)|dydt,\\
&&\leq 2^{-k/2}\frac{1}{\sqrt{n!2^{n}\sqrt{\pi}}}\int_{0}^{t}\int_{-1}^{1}\left|\frac{H'_{m+1}(y)}{m+1}\right|dydt,\\
&&\leq 2^{-k/2}\frac{1}{(\sqrt{n!2^{n}\sqrt{\pi}})(m+1)}\int_{0}^{t}\int_{-1}^{1}|H'_{m+1}(y)|dydt.
\end{eqnarray*}
By putting $\int_{-1}^{1}|H'_{m+1}(y)|dy = h$, we get
\begin{eqnarray*}\label{P2_con12}
|J^{2}\psi_{nm}(t)|\leq 2^{-k/2}\frac{1}{(\sqrt{n!2^{n}\sqrt{\pi}})(m+1)}\int_{0}^{t}h dt.
\end{eqnarray*}
Hence
\begin{eqnarray}\label{P2_con13}
|J^{2}\psi_{nm}(t)|\leq 2^{-k/2}\frac{1}{(\sqrt{n!2^{n}\sqrt{\pi}})(m+1)}h,
\end{eqnarray}
since $t \in [0,1]$. Now for $|c_{nm}|$, we have
\begin{eqnarray}\label{P2_con14}
&&c_{nm}=\int_{0}^{1}f(t)\psi_{nm}(t)dt,\\
&&\label{P2_con15}
|c_{nm}| \leq \int_{0}^{1}|f(t)||\psi_{nm}(t)|dt.
\end{eqnarray}
Now using \eqref{P2_con3}, we have
\begin{eqnarray*}\label{P2_con20}
|c_{nm}| \leq \eta\int_{0}^{1}|\psi_{nm}(t)|dt.
\end{eqnarray*}
By \eqref{P2_Def1}, we have
\begin{eqnarray*}\label{P2_con16}
|c_{nm}|\leq  2^{k/2}\frac{\eta}{\sqrt{n!2^{n}\sqrt{\pi}}}\int_{\frac{\hat{n}-1}{2^k}}^{\frac{\hat{n}+1}{2^k}}|H_{m}(2^{k}t-\hat{n})|dt.
\end{eqnarray*}
Now by change of variable $2^kt-\hat{n}=y$, we get
\begin{eqnarray*}\label{P2_con17}
&&|c_{nm}|\leq  2^{-k/2}\frac{\eta}{\sqrt{n!2^{n}\sqrt{\pi}}}\int_{-1}^{1}|H_{m}y|dy,\\
&&\label{P2_con18} |c_{nm}|\leq  2^{-k/2}\frac{\eta}{\sqrt{n!2^{n}\sqrt{\pi}}}\int_{-1}^{1}\left|\frac{H'_{m+1}(y)}{m+1}\right|dy.
\end{eqnarray*}
By putting $\int_{-1}^{1}|H'_{m+1}(y)|dy = h$, we have
\begin{eqnarray}\label{P2_con19}
|c_{nm}|\leq 2^{-k/2}\frac{1}{(\sqrt{n!2^{n}\sqrt{\pi}})(m+1)}\eta h.
\end{eqnarray}
Now using equation \eqref{P2_con13} and \eqref{P2_con19} in \eqref{P2_con9}
\begin{eqnarray}\label{P2_con21}
&&\|E_{k,M}\|_2^2\leq 2^{-2k}\eta^2h^4\sum^{\infty}_{n=2^k}\sum^{\infty}_{m=M}\sum^{\infty}_{s=2^k}\sum^{\infty}_{r=M}\int_{0}^{1}\frac{1}{(\sqrt{n!2^{n}\sqrt{\pi}})^2(m+1)^2}\frac{1}{(\sqrt{s!2^{s}\sqrt{\pi}})^2(r+1)^2} dt,\\
&&\label{P2_con22}
\|E_{k,M}\|_2^2\leq 2^{-2k}\eta^2h^4\sum^{\infty}_{n=2^k}\frac{1}{n!2^{n}\sqrt{\pi}}\sum^{\infty}_{s=2^k}\frac{1}{s!2^{s}\sqrt{\pi}} \sum^{\infty}_{m=M}\frac{1}{(m+1)^2}\sum^{\infty}_{r=M}\frac{1}{(r+1)^2}.
\end{eqnarray}
Here all four series converges and
$\|E_{k,M}\| \longrightarrow 0$ as $k,M \rightarrow \infty$.
\end{proof}
\begin{remark}
The above theorem can easily be extended for the method HeWQA.
\end{remark}
\begin{theorem}
Let us assume that $f(t)=\frac{d^u}{dt^2}\in L^2(\mathbb{R})$ is a continuous function on $[0,1]$ and its first derivative is bounded for all $t\in [0,1]$ there exists $\eta$ such that $\left|\frac{df}{dt}\right|\leq \eta$, then the method based on HWQA and HWNA converges.
\end{theorem}
\begin{proof}
The proof is similar to the previous theorem.
\end{proof}
\section{Numerical Illustrations}\label{P2_sec6} In this section we apply HeWQA, HeWNA, HWQA and HWNA on proposed model which occurs in exothermic reaction \eqref{P2_CHEM5}. We also solve four other examples from real life and compare solutions with among these four methods and with exact solutions whenever available.

To examine the accuracy of methods we define maximum absolute error $L_\infty$ as
\begin{equation}
L_\infty=\max_{x \in [0,1]}|y(t)-y_w(t)|
\end{equation}
here $y(t)$ is exact solution and $y_w(t)$ is wavelet solution.
and the $L_2$-norm error as
\begin{equation}
L_2=\left(\sum_{j=o}^{2M}|y(x_j)-y_w(x_j)|^2\right)^{1/2}
\end{equation}
here $y(x_j)$ is exact solution and $y_w(x_j)$ is wavelet solution at the point $x_j$.

\subsection{Example 1 (Exothermic Reaction with Modified Arrehenius Law)}\label{P2_sec6a}
Consider the non-linear SBVP \eqref{P2_CHEM5} with given boundary condition
\begin{eqnarray}\label{P2_61}
Ly+B\exp\left(\frac{-A}{(c^n+y^{n})^{1/n}}\right)=0,\quad y'(0)=0,\quad y(1)=0,
\end{eqnarray}
we take some particular cases when $A=1,B=1,c=1$.

Comparison Graphs taking initial vector [$0,0,\hdots,0$] and $J=1$, $J=2$ with $n=1,~k_g=1$; $n=1,~k_g=2$; $n=2,~k_g=1$; $n=2,~k_g=2$;$n=3,~k_g=1$ and $n=3,~k_g=2$ are plotted in Figure \ref{exfig1}, Figure \ref{exfig2}, Figure \ref{exfig3}, Figure \ref{exfig4}, Figure \ref{exfig5} and Figure \ref{exfig6} respectively. Tables for solution is tabulated in table \ref{extab1}, \ref{extab2}, table \ref{extab3}, table \ref{extab4}, table \ref{extab5} and table \ref{extab6} respectively.
\begin{table}[H]
\caption{\small{ Comparison of HWQA, HeWNA, HWQA, HeWQA method solution for example \ref{P2_sec6a} with $n=1,~k_g=1$ taking $J=2$ :}}\label{extab1}								 
\begin{center}
\resizebox{10cm}{2cm}{\begin{tabular}{|c | l| l| l| l|}
\hline Grid Points	&	HWNA \cite{akvdt2018}	&	HeWNA	&	HWQA \cite{akvdt2018}	&	HeWQA\\
\hline
   0	&	0.098471606	&	0.098471868	&	0.099733232	&	0.098721649	\\
  1/16	&	0.098078784	&	0.09807895	&	0.099334324	&	0.098328004	\\
  3/16	&	0.094937743	&	0.094937908	&	0.096144758	&	0.095180981	\\
  5/16	&	0.08866797	&	0.08866813	&	0.089779277	&	0.088898418	\\
  7/16	&	0.079294153	&	0.079294305	&	0.080265493	&	0.079503596	\\
  9/16	&	0.066853514	&	0.066853646	&	0.067645598	&	0.06703223	\\
 11/16	&	0.051396022	&	0.051396117	&	0.051977236	&	0.051533416	\\
 13/16	&	0.032984684	&	0.032984721	&	0.033334595	&	0.033070901	\\
 15/16	&	0.011695904	&	0.011695851	&	0.011809628	&	0.01172465	
\\\hline
\end{tabular}}											
\end{center}											
\end{table}	

\begin{figure}[H]
\begin{center}
\includegraphics[scale=0.30]{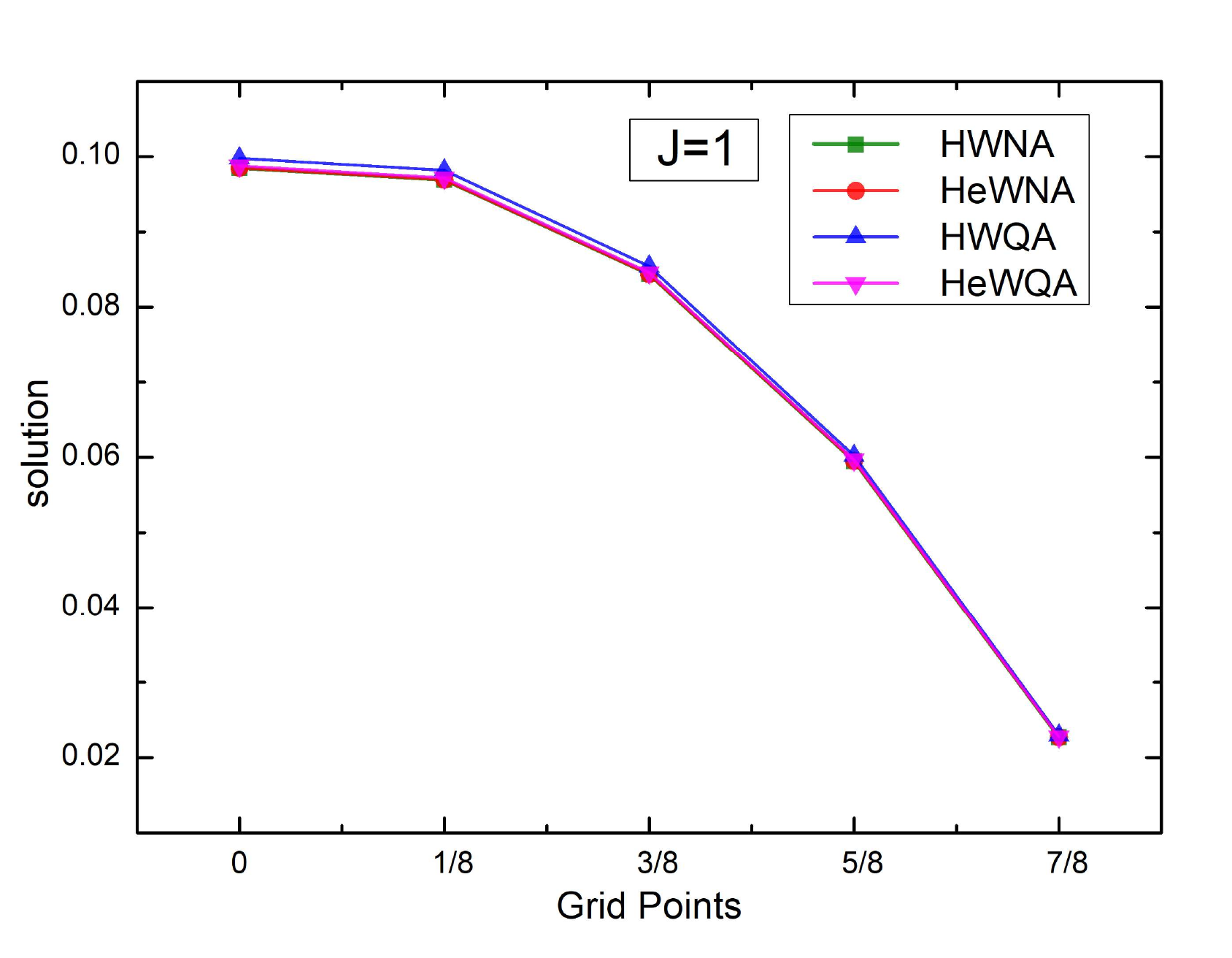}\includegraphics[scale=0.30]{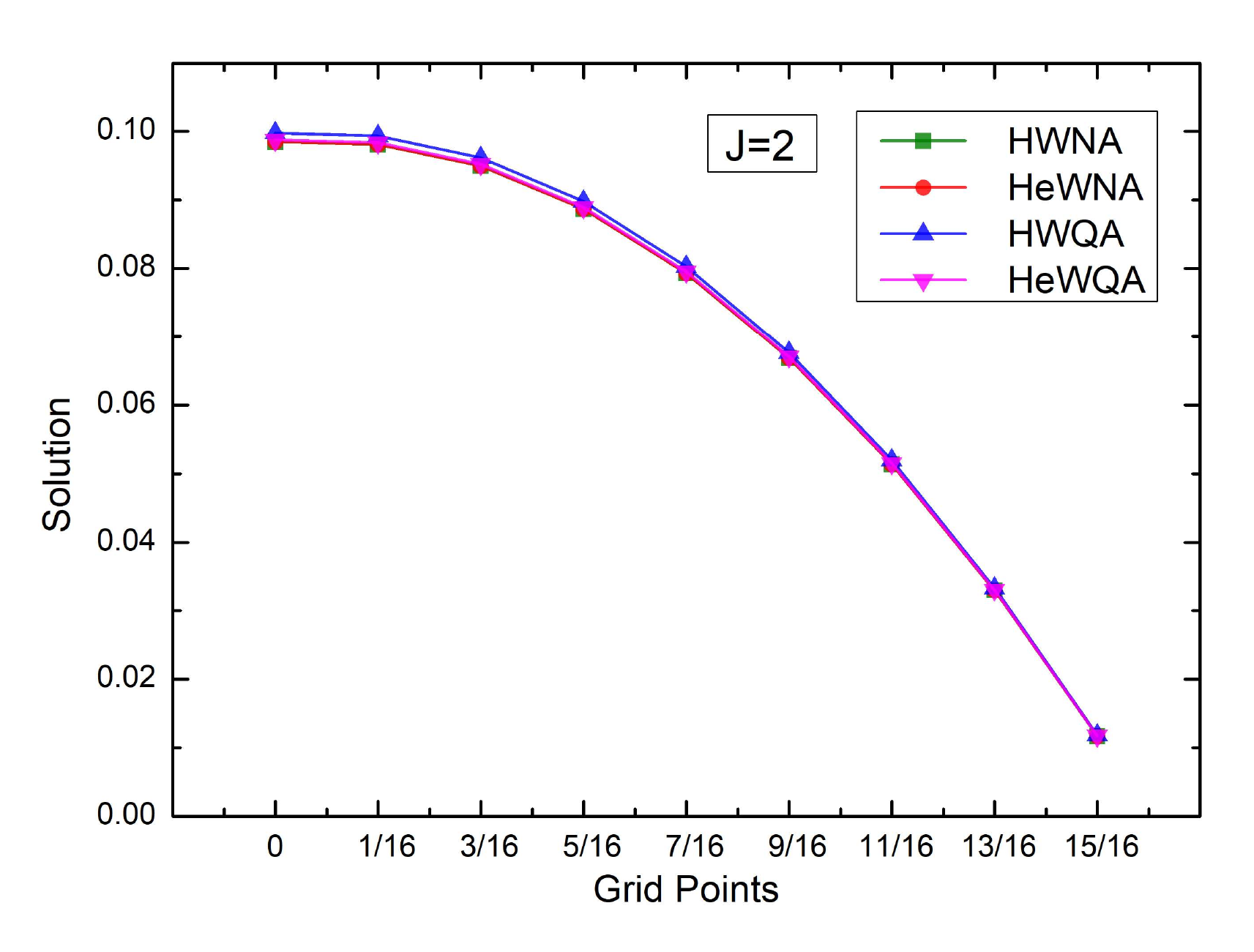}
\end{center}
\caption{Comparison plots of solution methods for $J=1,2$ for example \ref{P2_sec6a} with $n=1,~k_g=1$:}\label{exfig1}
\end{figure}
		
\begin{table}[H]
\caption{\small{ Comparison of HWQA, HeWNA, HWQA, HeWQA method solution for example \ref{P2_sec6a} with $n=1,~k_g=2$ taking $J=2$ :}}\label{extab2}														 \centering											
\begin{center}	
\resizebox{10cm}{2cm}{\begin{tabular}{|c | l| l| l| l|}
\hline Grid Points	&	   HWNA \cite{akvdt2018}	&	HeWNA	&	  HWQA \cite{akvdt2018}	&	HeWQA	\\
\hline0	&	0.063984889	&	0.063981865	&	0.064360996	&	0.064044224	\\
  1/16	&	0.063730596	&	0.063727533	&	0.064104699	&	0.063789741	\\
  3/16	&	0.061697078	&	0.061694111	&	0.062055233	&	0.06175505	\\
  5/16	&	0.057636632	&	0.057633858	&	0.057963614	&	0.057691954	\\
  7/16	&	0.05156246	&	0.051559975	&	0.051844545	&	0.051613138	\\
  9/16	&	0.043494423	&	0.04349232	&	0.043720272	&	0.043537903	\\
 11/16	&	0.0334591	&	0.033457472	&	0.033620792	&	0.033492504	\\
 13/16	&	0.02148988	&	0.021488815	&	0.021584087	&	0.021510584	\\
 15/16	&	0.007627057	&	0.007626644	&	0.007656348	&	0.007633719	
\\\hline
\end{tabular}}											
\end{center}											
\end{table}

\begin{figure}[H]
\begin{center}
\includegraphics[scale=0.30]{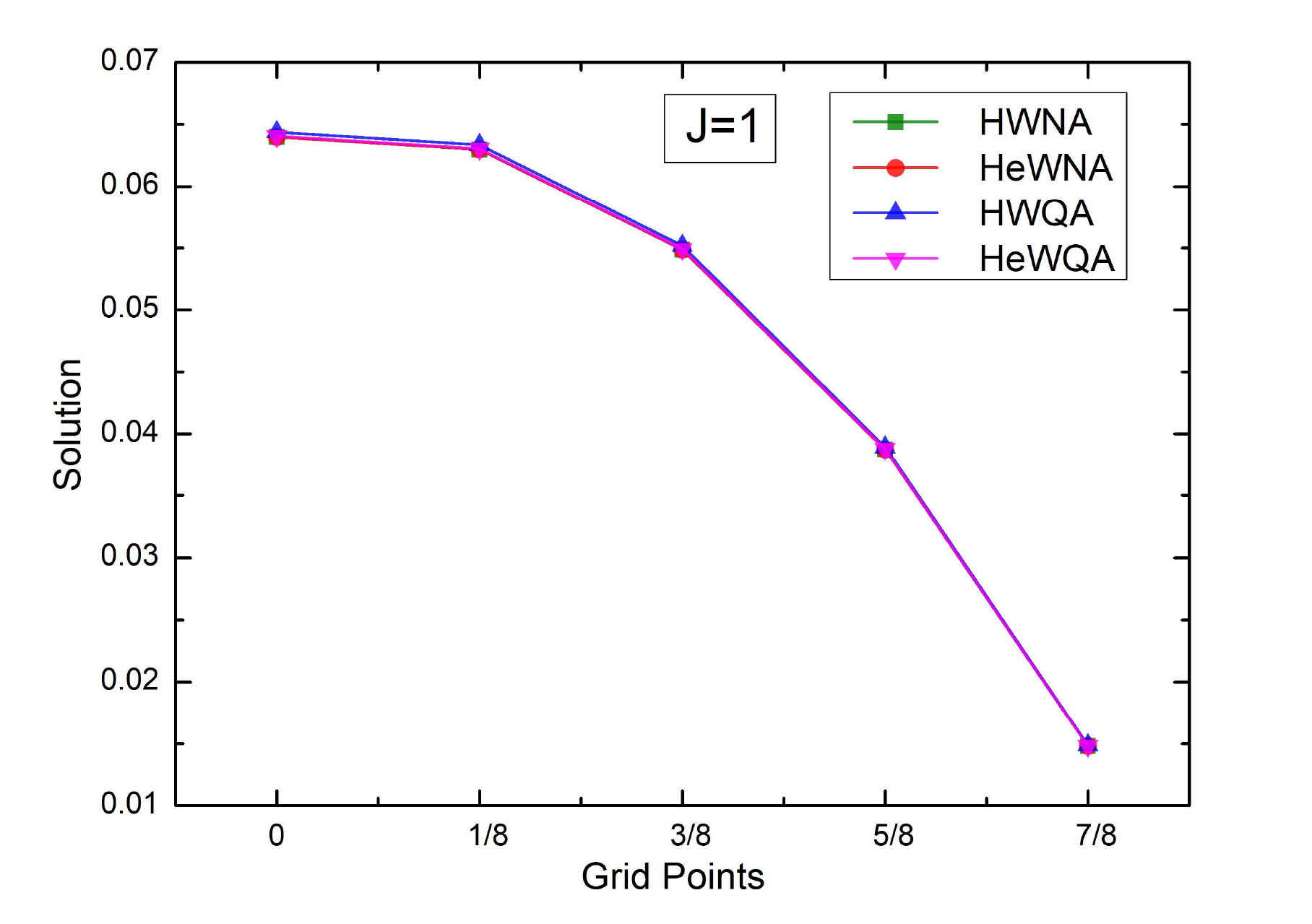}\includegraphics[scale=0.30]{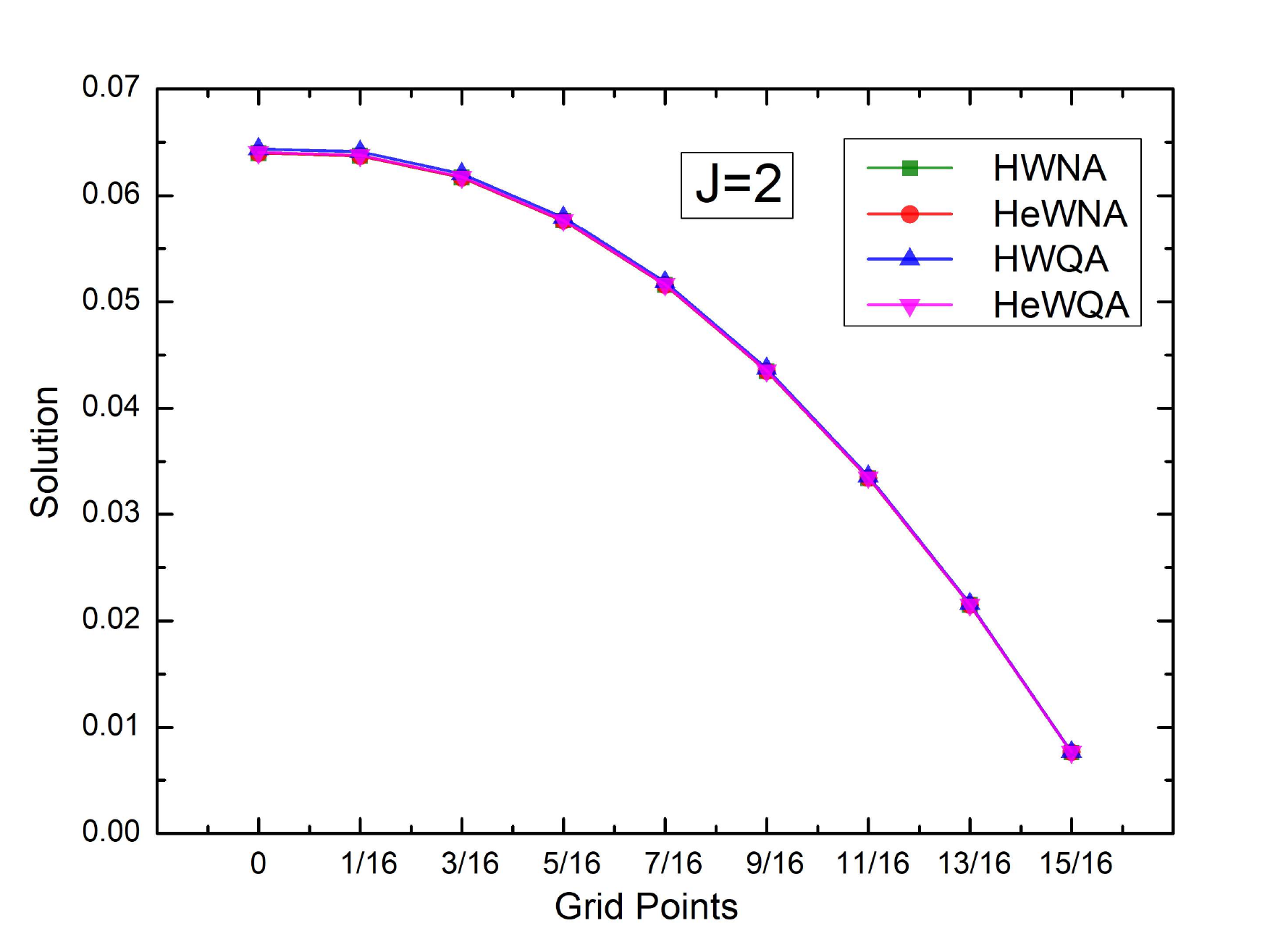}
\end{center}
\caption{Comparison plots of solution methods for $J=1,2$ for example \ref{P2_sec6a} with $n=1,~k_g=2$:}\label{exfig2}
\end{figure}

\begin{table}[H]
\caption{\small{Comparison of HWQA, HeWNA, HWQA, HeWQA method solution for example \ref{P2_sec6a} with $n=2,~k_g=1$ taking $J=2$ :}}\label{extab3} 													 
\begin{center}	
\resizebox{10cm}{2cm}{\begin{tabular}{|c | l| l| l| l|}
\hline		
Grid Points	&	  HWNA \cite{akvdt2018}	&	HeWNA	&	HWQA \cite{akvdt2018}	&	HeWQA	\\
\hline
   0	&	0.092208342	&	0.092207929	&	0.092312756	&	0.092313207	\\
  1/16	&	0.091847576	&	0.091847154	&	0.091951987	&	0.091952438	\\
  3/16	&	0.088961585	&	0.088961165	&	0.089065836	&	0.089066284	\\
  5/16	&	0.083190705	&	0.08319029	&	0.083293585	&	0.083294024	\\
  7/16	&	0.07453702	&	0.074536618	&	0.074635444	&	0.074635864	\\
  9/16	&	0.063003362	&	0.06300299	&	0.063092011	&	0.063092397	\\
 11/16	&	0.048592958	&	0.048592641	&	0.048664618	&	0.048664943	\\
 13/16	&	0.031308953	&	0.031308728	&	0.031355798	&	0.031356022	\\
 15/16	&	0.011153812	&	0.011153735	&	0.01116988	&	0.011169952	
\\\hline
\end{tabular}}											
\end{center}											
\end{table}			

\begin{figure}[H]
\begin{center}
\includegraphics[scale=0.30]{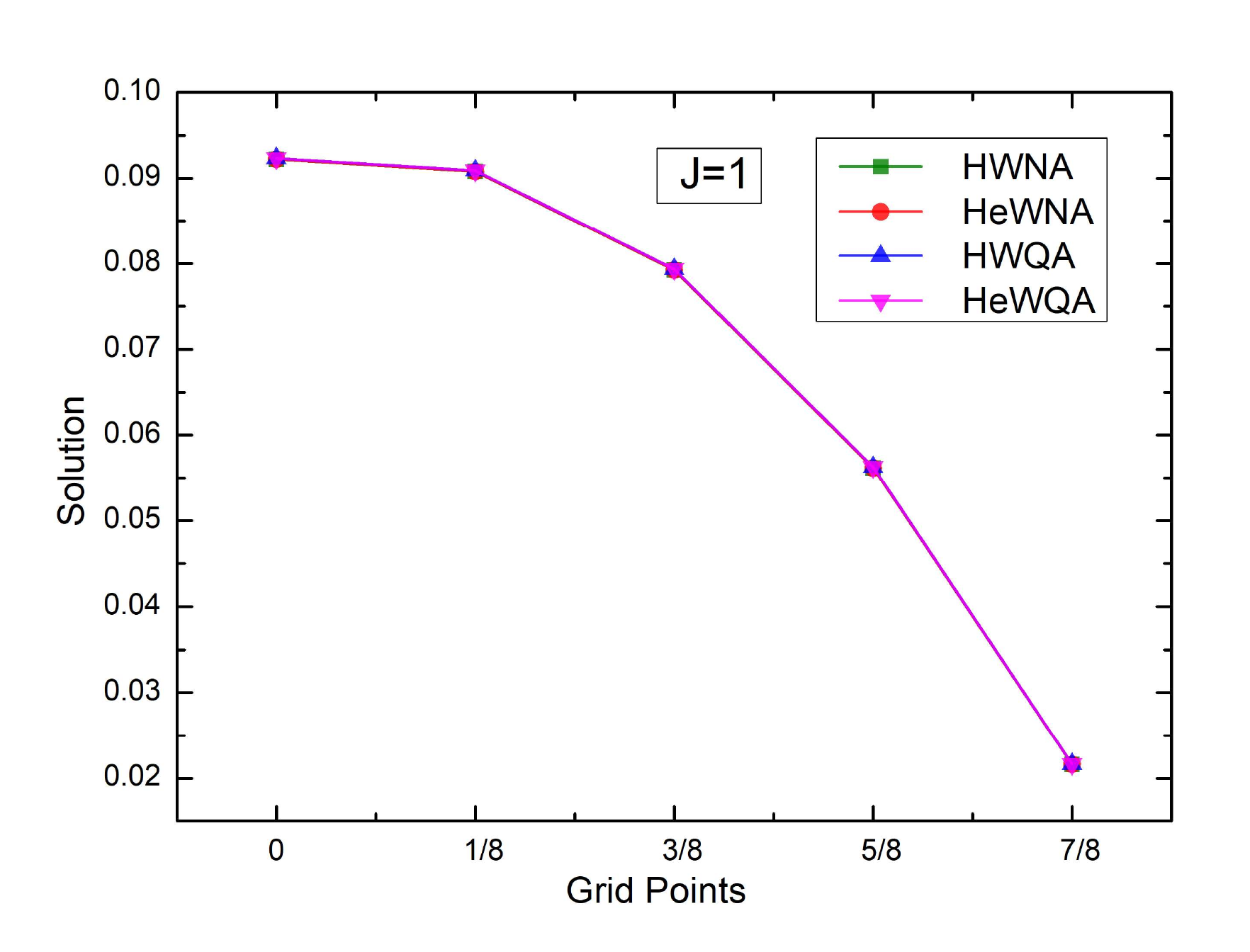}\includegraphics[scale=0.30]{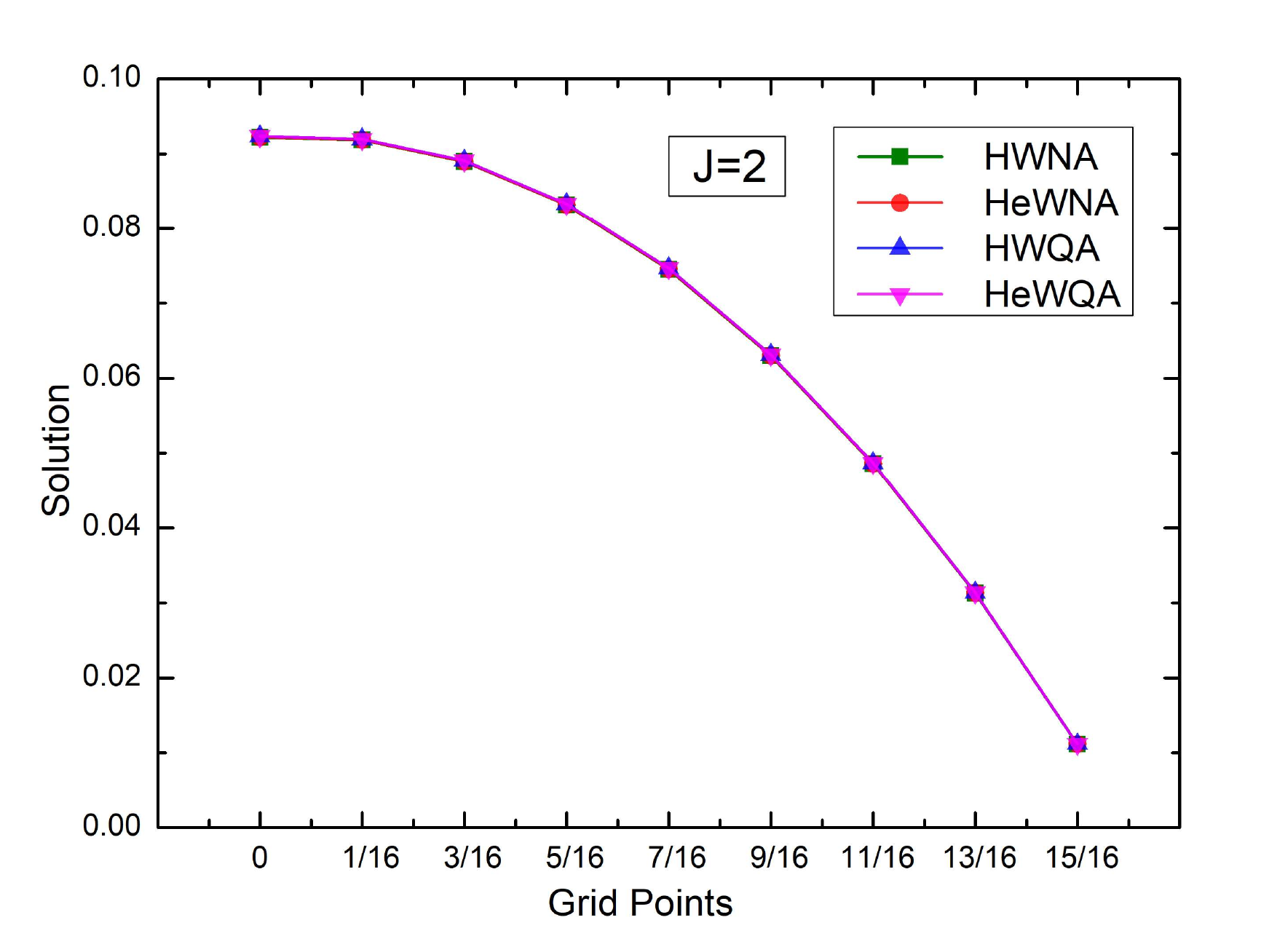}
\end{center}
\caption{Comparison plots of solution methods for example \ref{P2_sec6a} with $n=2,~k_g=1$:}\label{exfig3}
\end{figure}	

\begin{table}[H]
\caption{\small{Comparison of HWQA, HeWNA, HWQA, HeWQA method solution for example \ref{P2_sec6a} with $n=2,~k_g=2$ taking $J=2$ :}}\label{extab4}							 \begin{center}	
\resizebox{10cm}{2cm}{\begin{tabular}{|c | l| l| l| l|}
\hline
Grid Points	&	  HWNA \cite{akvdt2018}	&	HeWNA	&	HWQA \cite{akvdt2018}	&	HeWQA	\\
\hline
   0	&	0.061376073	&	0.061376325	&	0.061411277	&	0.061411345	\\
  1/16	&	0.061136122	&	0.061135877	&	0.061171325	&	0.061171393	\\
  3/16	&	0.059216558	&	0.05921632	&	0.05925171	&	0.059251777	\\
  5/16	&	0.055377823	&	0.055377598	&	0.055412494	&	0.055412559	\\
  7/16	&	0.049620654	&	0.049620448	&	0.04965375	&	0.049653812	\\
  9/16	&	0.041946044	&	0.041945866	&	0.0419757	&	0.041975756	\\
 11/16	&	0.032355109	&	0.032354967	&	0.032378845	&	0.032378891	\\
 13/16	&	0.020848901	&	0.020848809	&	0.020864147	&	0.020864177	\\
 15/16	&	0.007428191	&	0.007428161	&	0.007433258	&	0.007433263	
\\\hline
\end{tabular}}											
\end{center}											
\end{table}
	
\begin{figure}[H]
\begin{center}
\includegraphics[scale=0.30]{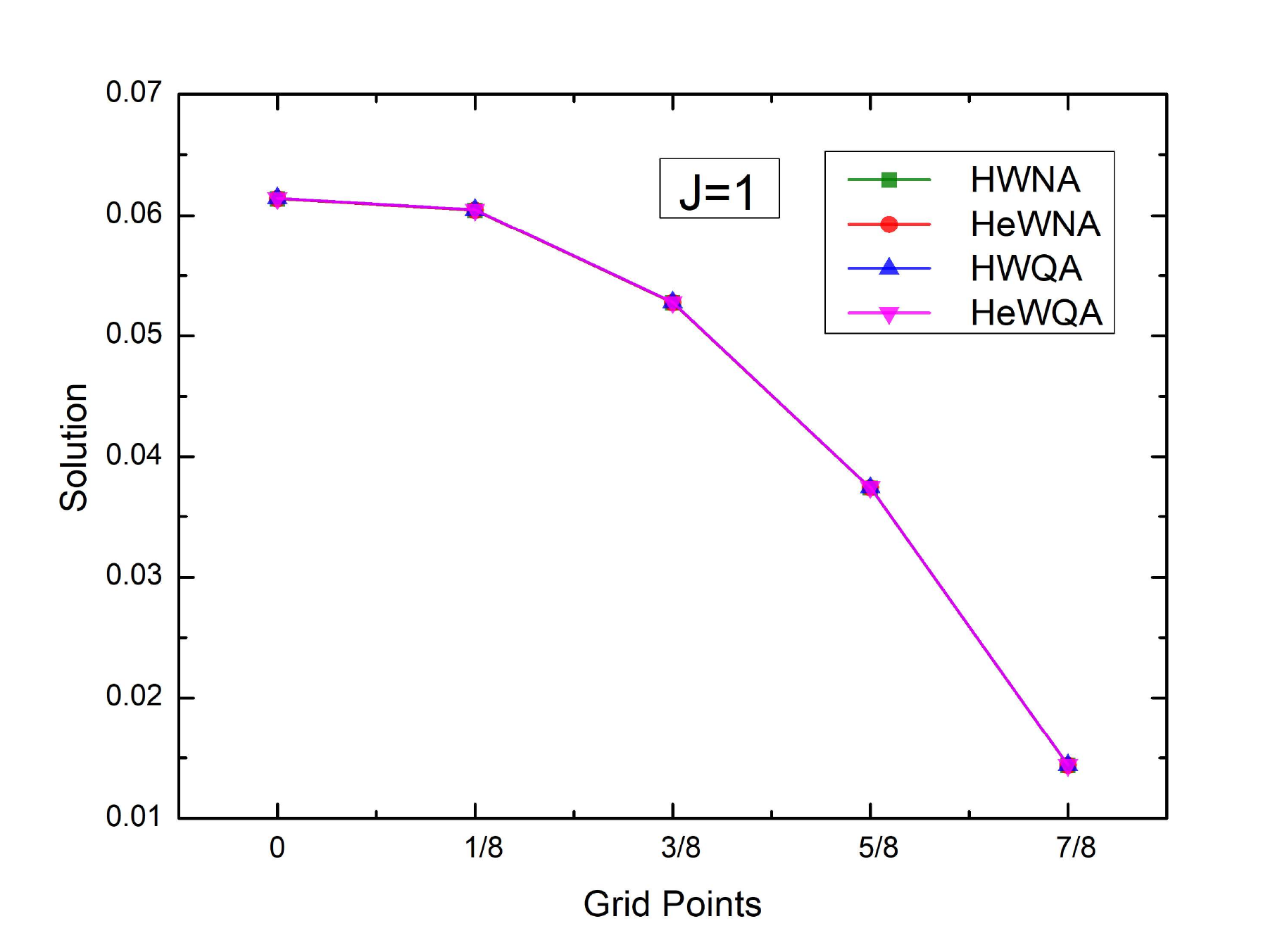}\includegraphics[scale=0.30]{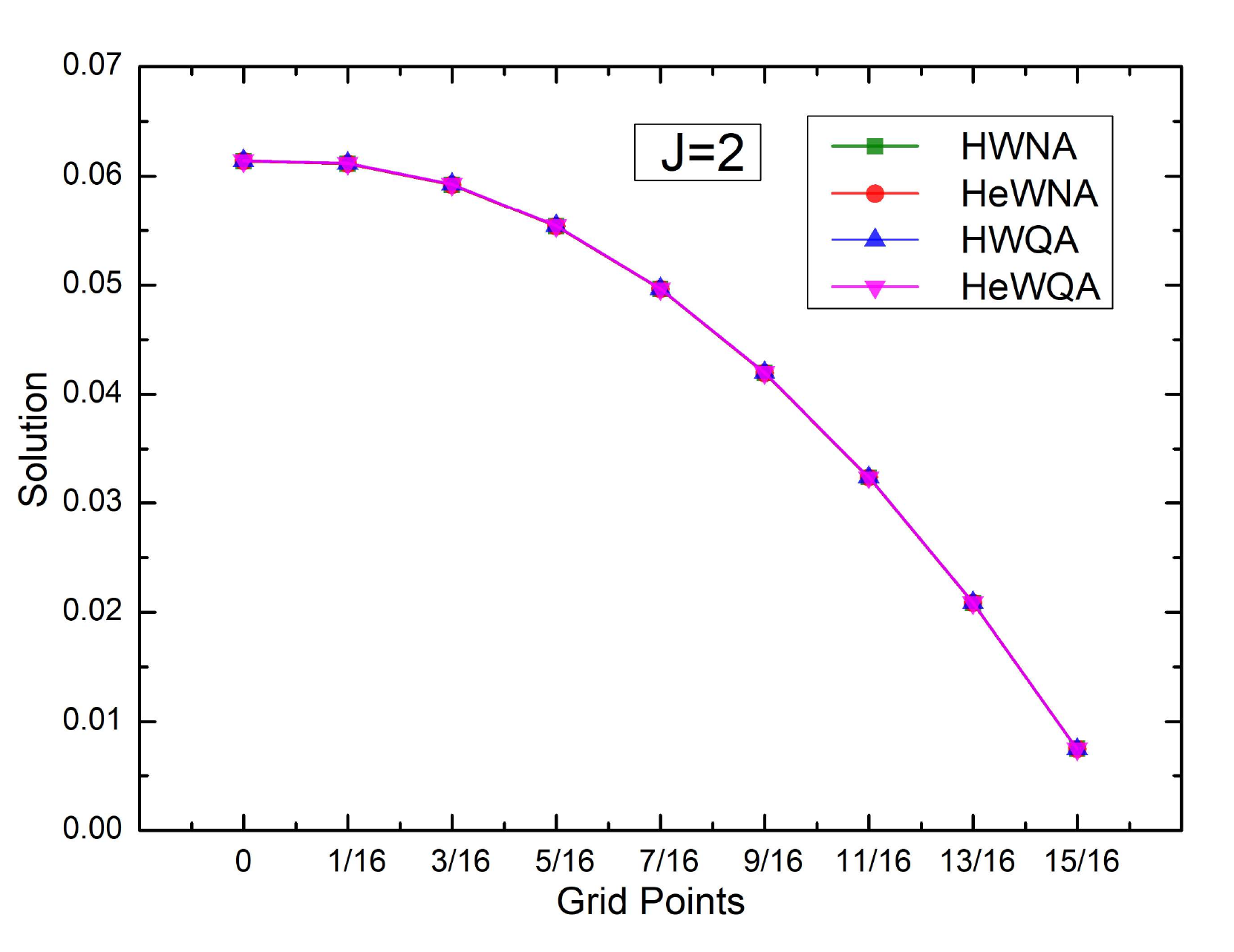}
\end{center}
\caption{Comparison plots of solution methods for $J=1,2$ for example \ref{P2_sec6a} with $n=2,~k_g=2$}\label{exfig4}
\end{figure}

\begin{table}[H]
\caption{\small{Comparison of HWQA, HeWNA, HWQA, HeWQA method solution for example \ref{P2_sec6a} with $n=3,~k_g=1$ taking $J=2$ :}}\label{extab5}				
\begin{center}	
\resizebox{10cm}{2cm}{
\begin{tabular}{|c | l| l| l| l| }
\hline		
Grid Points	&	 HWNA \cite{akvdt2018}	&	 HeWNA	&	 HWQA \cite{akvdt2018}	&	 HeWQA	\\
\hline
   0	&	0.091982324	&	0.091982282	&	0.091998157	&	0.091998241	\\
  1/16	&	0.091622975	&	0.091622932	&	0.091622975	&	0.091638892	\\
  3/16	&	0.088748193	&	0.08874815	&	0.088748193	&	0.088764086	\\
  5/16	&	0.082998726	&	0.082998685	&	0.082998726	&	0.083014381	\\
  7/16	&	0.074374752	&	0.074374713	&	0.074374752	&	0.07438963	\\
  9/16	&	0.062876486	&	0.062876451	&	0.062876486	&	0.062889712	\\
 11/16	&	0.048504139	&	0.048504111	&	0.048504139	&	0.048514619	\\
 13/16	&	0.031257869	&	0.031257851	&	0.031257869	&	0.031264554	\\
 15/16	&	0.011137751	&	0.011137745	&	0.011137751	&	0.011139994	\\\hline
\end{tabular}}											
\end{center}									
\end{table}
	
\begin{figure}[H]
\begin{center}
\includegraphics[scale=0.30]{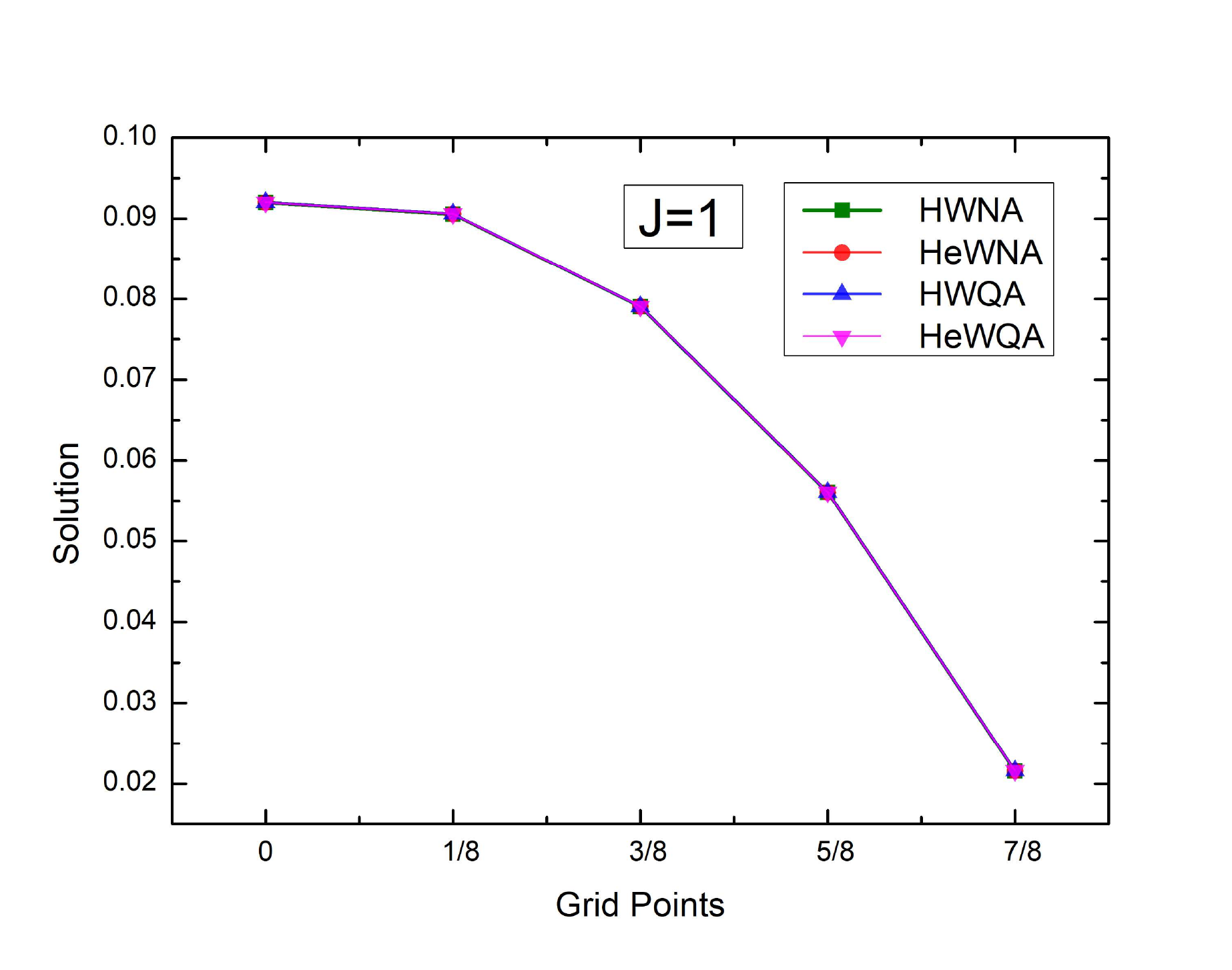}\includegraphics[scale=0.30]{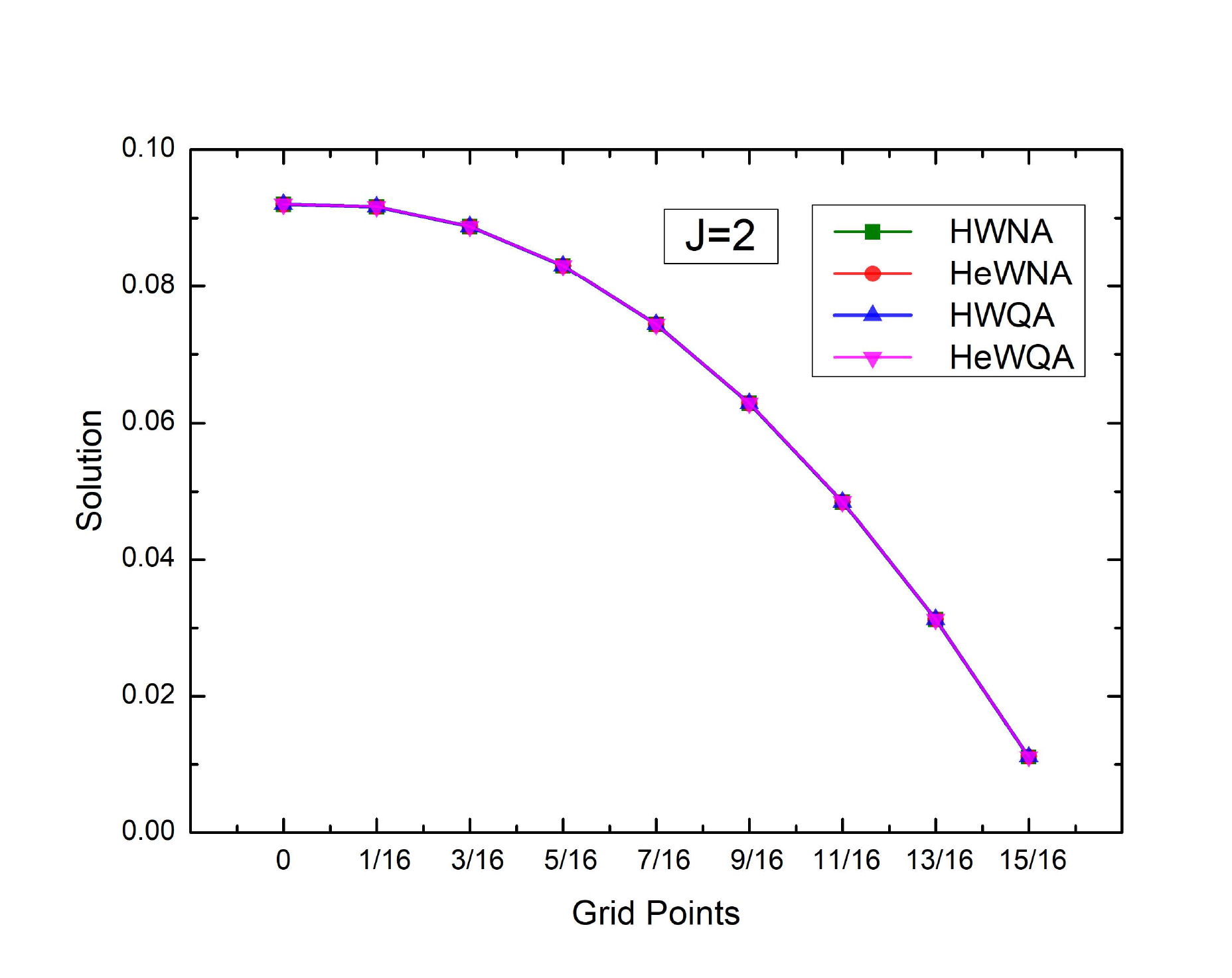}
\end{center}
\caption{Comparison plots of solution methods for $J=1,2$ for example \ref{P2_sec6a} with $n=3,~k_g=1$}\label{exfig5}
\end{figure}	

\begin{table}[H]
\caption{\small{Comparison of HWQA, HeWNA, HWQA, HeWQA method solution for example \ref{P2_sec6a} with $n=3,~k_g=2$ taking $J=2$:}}\label{extab6}					 
\begin{center}	
\resizebox{10cm}{2cm}{
\begin{tabular}{|c | l| l| l| l| }
\hline		
Grid Points	&	 HWNA \cite{akvdt2018}	&	 HeWNA	&	 HWQA \cite{akvdt2018}	&	 HeWQA	\\
\hline
   0	&	0.061315351	&	0.061315829	&	0.061318845	&	0.061321705	\\
  1/16	&	0.061075828	&	0.061075814	&	0.061079322	&	0.061083071	\\
  3/16	&	0.059159646	&	0.059159634	&	0.059163135	&	0.059166226	\\
  5/16	&	0.055327307	&	0.055327295	&	0.055330739	&	0.055333256	\\
  7/16	&	0.049578851	&	0.04957884	&	0.049582101	&	0.049584212	\\
  9/16	&	0.041914328	&	0.041914319	&	0.041917195	&	0.041918897	\\
 11/16	&	0.032333785	&	0.032333778	&	0.032336026	&	0.032337266	\\
 13/16	&	0.020837254	&	0.02083725	&	0.020838652	&	0.020839592	\\
 15/16	&	0.007424747	&	0.007424746	&	0.007425199	&	0.007424945	
\\\hline
\end{tabular}}											
\end{center}											
\end{table}	

\begin{figure}[H]
\begin{center}
\includegraphics[scale=0.30]{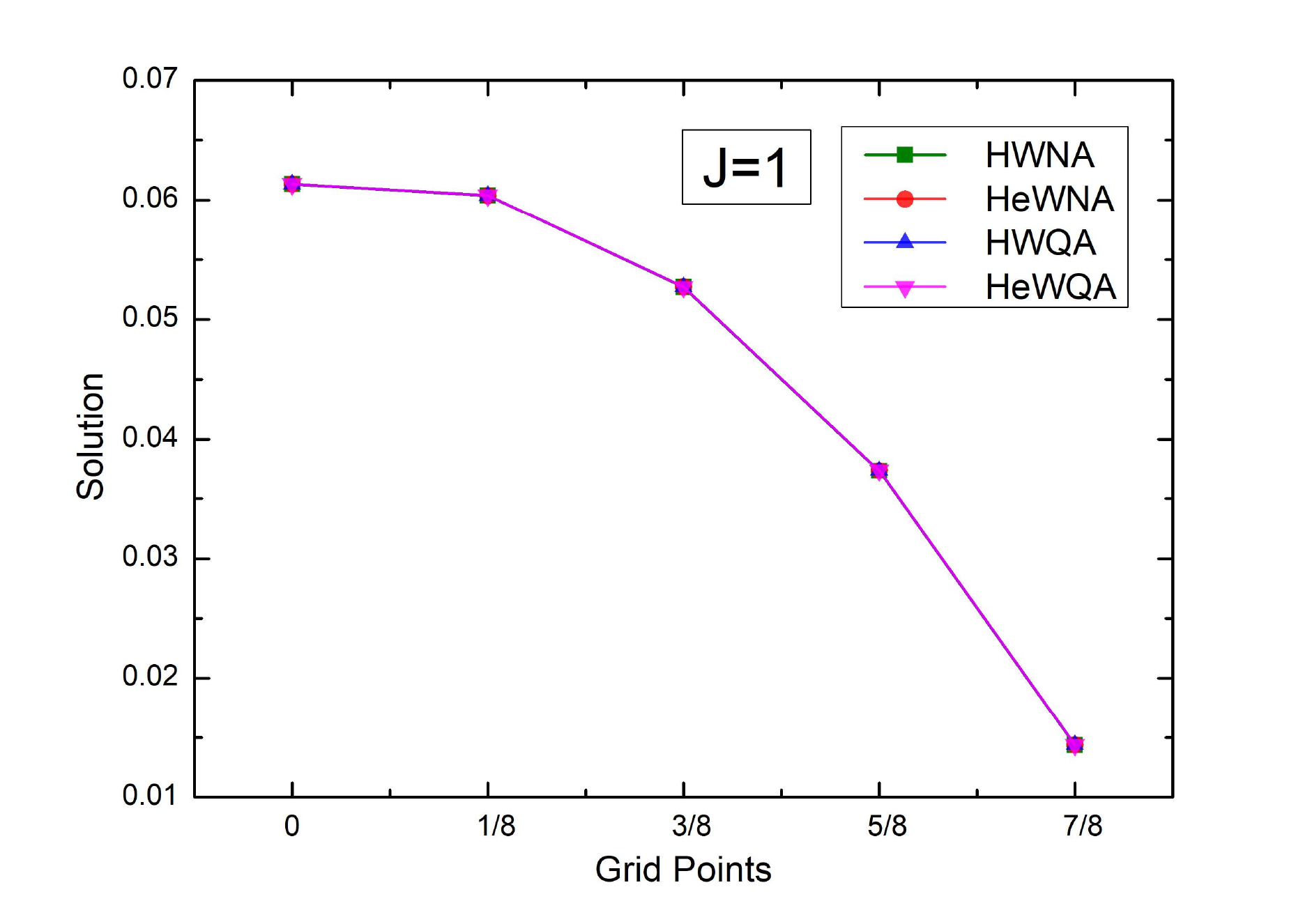}\includegraphics[scale=0.30]{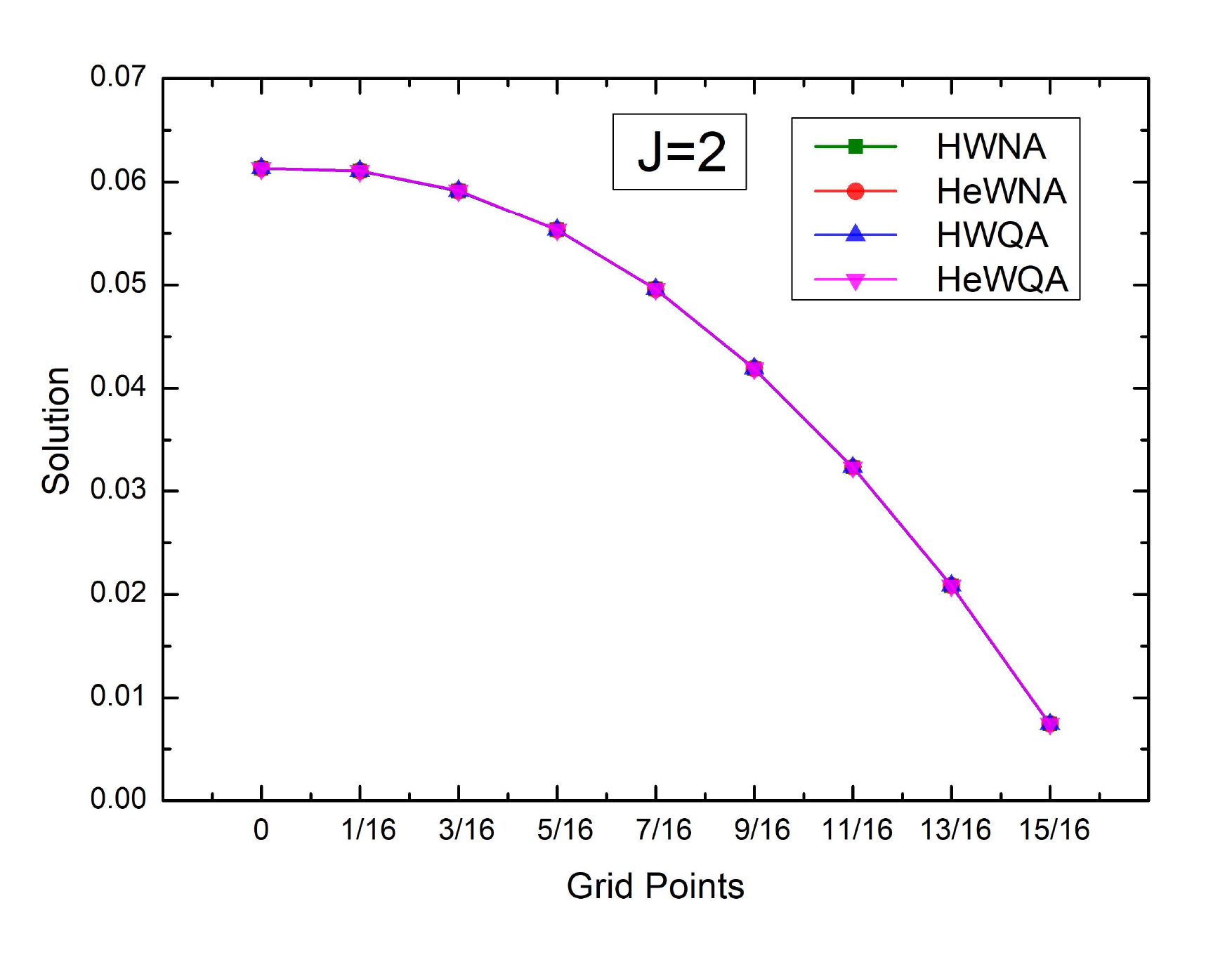}
\end{center}
\caption{Comparison plots of solution methods for $J=1,2$ for example \ref{P2_sec6a} with $n=3,~k_g=2$}\label{exfig6}
\end{figure}	
The example defined by equation \eqref{P2_61} is new and does not exist in literature. So we are not in a situation to compare the results. We have considered $n=1,2,3$ and $k_g=1,2$. Tables \ref{extab1}, \ref{extab2}, \ref{extab3}, \ref{extab4}, \ref{extab5} , \ref{extab6} and figures \ref{exfig1},  \ref{exfig2}, \ref{exfig3},  \ref{exfig4}, \ref{exfig5}, \ref{exfig6} demonstrate the behaviour of the solution for $J=1,2$. HWNA, HeWNA, HWQA and HeWQA all give numerics which are very well comparable and shows that our proposed techniques are working well.

We also observed for small changes in initial vector, for example taking $[0.1,0.1,\hdots,0.1]$ or $[0.2,0.2,\hdots,0.2]$ doesn't significantly change the solution in any case.

\subsection{Example 2 (Stellar Structure)}\label{P2_sec6b}
Consider the non-linear SBVP:
\begin{eqnarray}\label{P2_62}
Ly(t)+y^{5}(t)=0,\quad y'(0)=0,\quad y(1)=\sqrt{\frac{3}{4}},\quad k_g=2,
\end{eqnarray}
Chandrasekhar (\cite{CS1967}, p88 ) derived above two point nonlinear SBVP. This equation arise in study of stellar structure. Its exact solution is $y(t)=\sqrt{\frac{3}{3+x^2}}$.

Comparison Graphs taking initial vector [$\sqrt{\frac{3}{4}},\sqrt{\frac{3}{4}},\hdots, \sqrt{\frac{3}{4}}$] and $J=1$, $J=2$ are plotted in Figure \ref{exfig7}. Tables for solution and error are tabulated in table \ref{extab7} and table \ref{exerr1}.
\begin{table}[H]
\caption{\small{Comparison of HWNA, HeWNA, HWQA, HeWQA methods solution with analytical solution for example \ref{P2_sec6b} taking $J=2$ :}}\label{extab7}											 \centering											
\begin{center}											
\resizebox{12cm}{2cm}{											
\begin{tabular}	{|c | l|  l|  l| l| l| }										
\hline      																					
Grid Points	&	  HWNA \cite{akvdt2018}	&	HeWNA	&	  HWQA \cite{akvdt2018}	&	HeWQA	&	   Exact\\	\hline
0	&	1.00023666	&	0.999999992	&	1.00023666	&	0.999999992	&	1	\\
  1/16	&	0.999586961	&	0.99934958	&	0.999586961	&	0.99934958	&	 0.999349593	\\
  3/16	&	0.994419294	&	0.994191616	&	0.994419294	&	0.994191616	&	 0.994191626	\\
  5/16	&	0.984319576	&	0.984110835	&	0.984319576	&	0.984110835	&	 0.984110842	\\
  7/16	&	0.969730094	&	0.96954859	&	0.969730094	&	0.96954859	&	 0.969548596	\\
  9/16	&	0.9512486	&	0.951101273	&	0.9512486	&	0.951101273	&	 0.951101277	\\
 11/16	&	0.92956584	&	0.92945791	&	0.92956584	&	0.92945791	&	 0.929457914	\\
 13/16	&	0.905403371	&	0.905338132	&	0.905403371	&	0.905338132	&	 0.905338136	\\
 15/16	&	0.879460746	&	0.879439538	&	0.879460746	&	0.879439538	&	 0.879439536
	\\\hline
\end{tabular}}											
\end{center}
\end{table}
\begin{table}[H]
\caption{\small{Comparison of error of HWNA, HeWNA, HWQA, HeWQA methods for example \ref{P2_sec6b} taking $J=2$ :}}\label{exerr1}											 \centering											
\begin{center}											
	\resizebox{8cm}{0.8cm}{											
		\begin{tabular}	{|c | l|  l|  l| l| }										
			\hline      																Error	&	  HWNA \cite{akvdt2018}	&	HeWNA	&	  HWQA \cite{akvdt2018}	&	HeWQA\\ \hline
			$L_\infty$	&	0.000237368	& 2.49669$\times 10^{-9}$  	&	0.000237368	&	2.49669$\times 10^{-9}$\\
			$L_2$	&	0.000471959	&	1.97638$\times 10^{-8}$	&  0.000471959		&		1.97638$\times 10^{-8}$ \\\hline
	\end{tabular}}											
\end{center}
\end{table}
\begin{figure}[H]
\begin{center}
\includegraphics[scale=0.30]{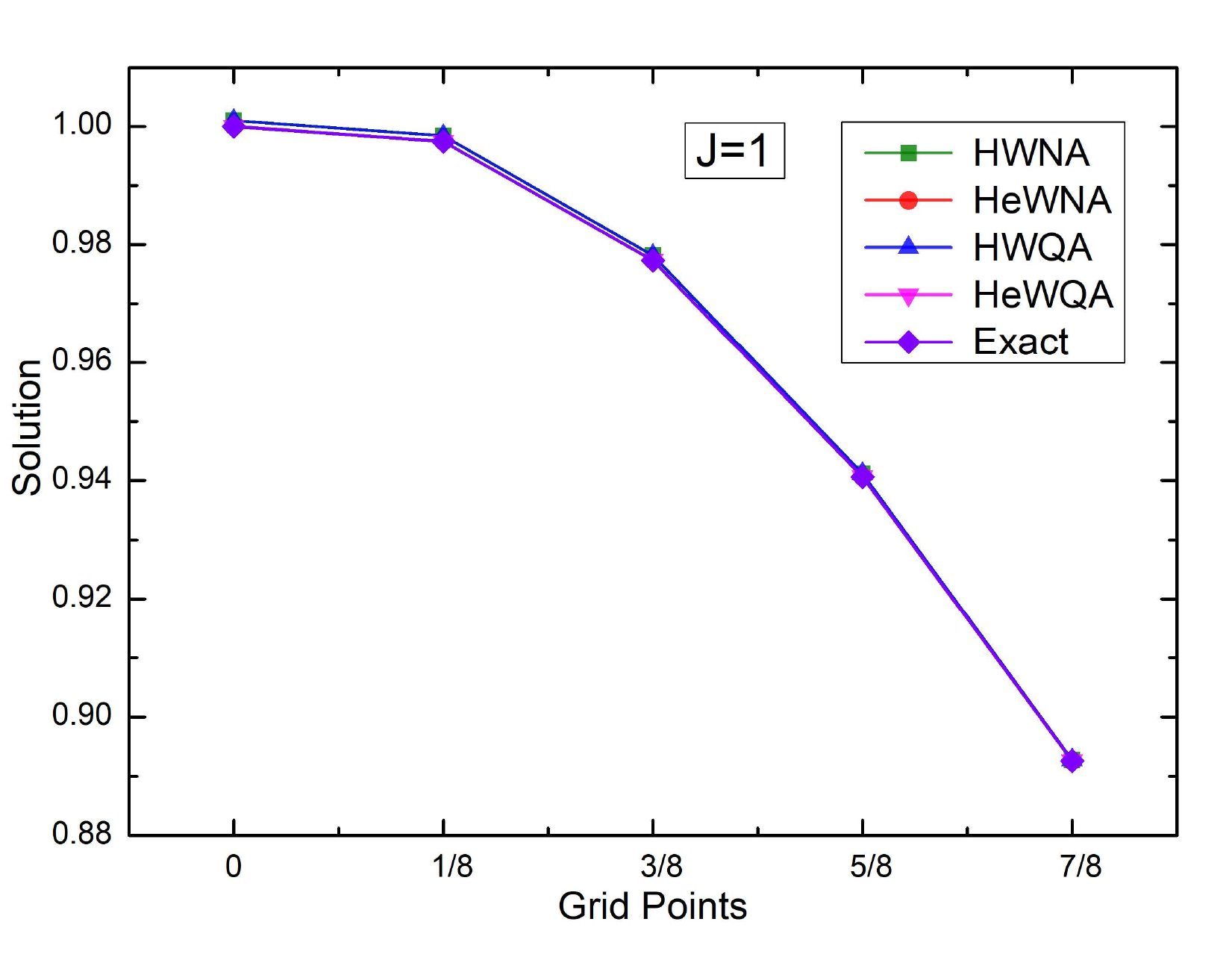}\includegraphics[scale=0.30]{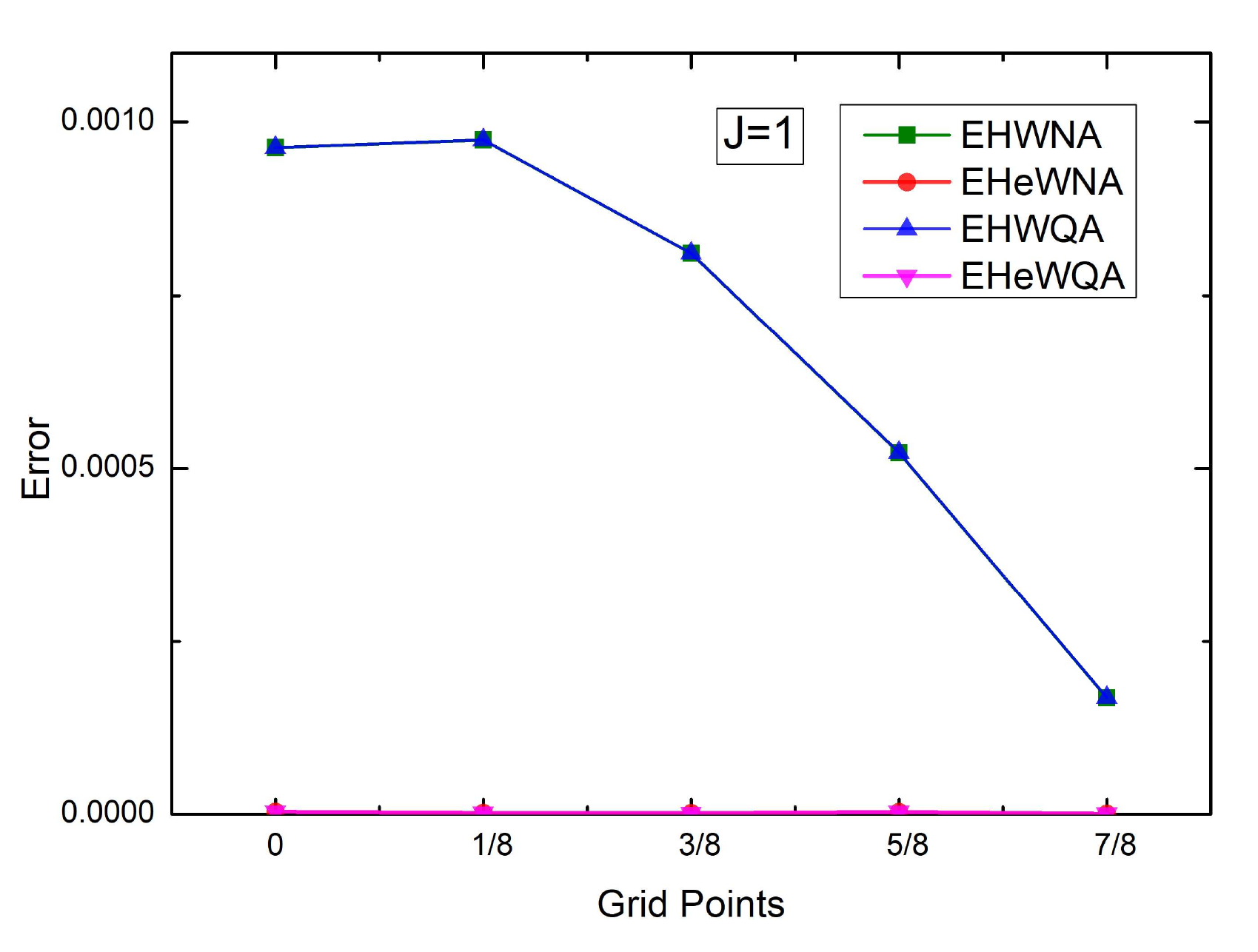}
\end{center}
\begin{center}
\includegraphics[scale=0.30]{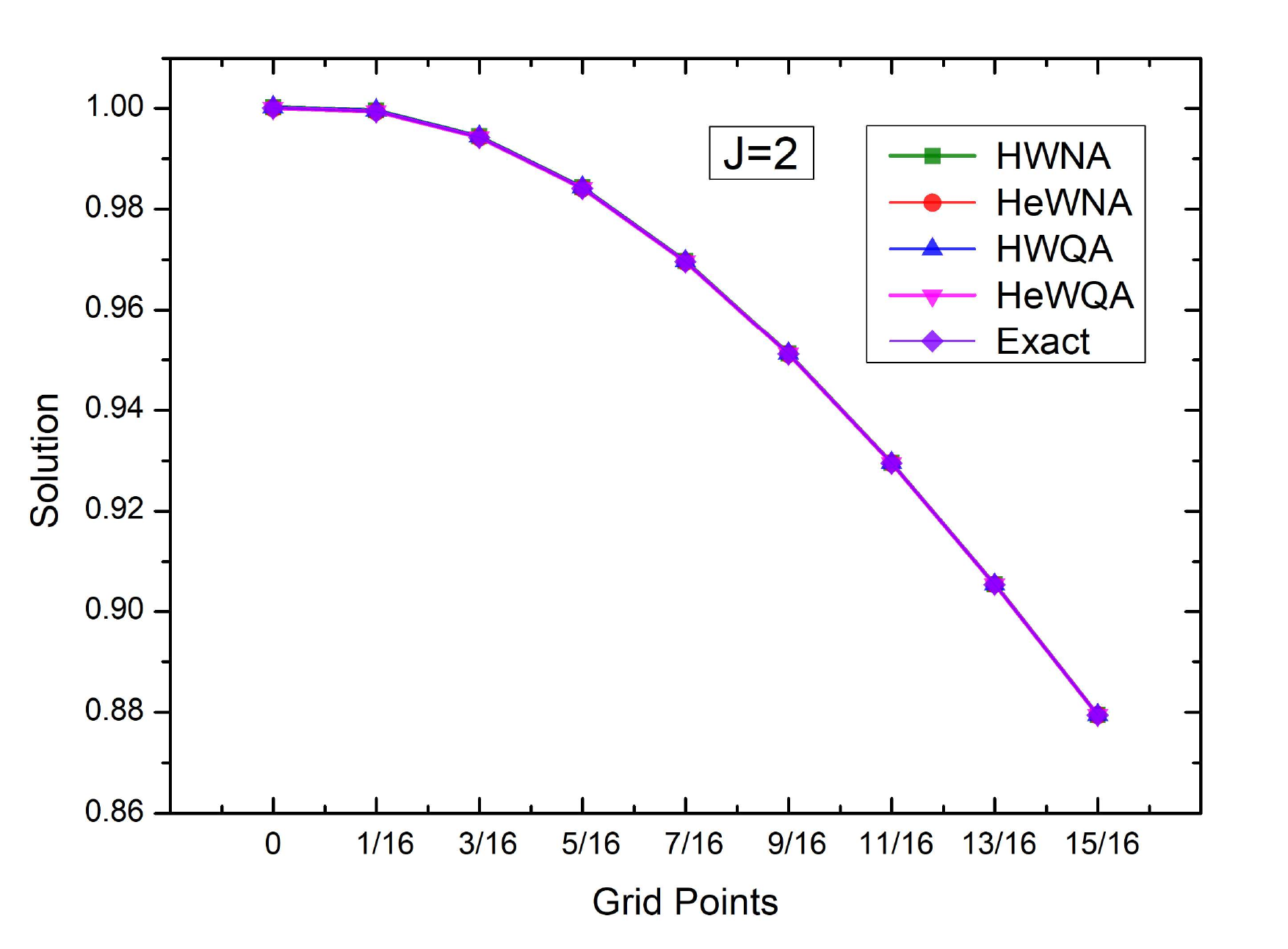}\includegraphics[scale=0.30]{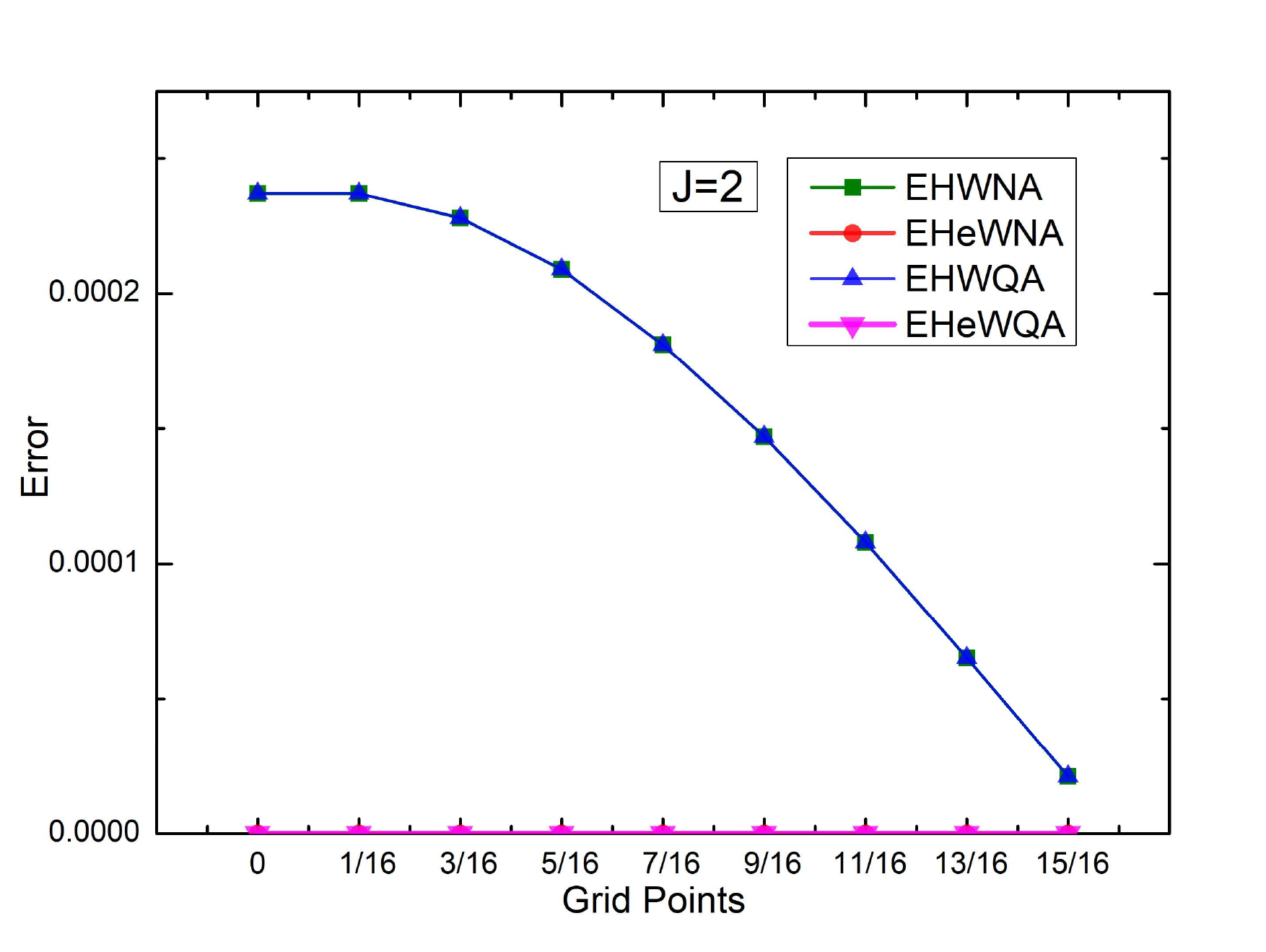}
\end{center}
\caption{Comparison plots and error plots of solution methods for $J=1,2$ for example \ref{P2_sec6b}}\label{exfig7}
\end{figure}
In this test case since exact solution of the SBVP governed by \eqref{P2_62} exists, We have compared our solutions with exact solution in table \ref{extab7} and figure \ref{exfig7}. Numerics again prove that method gives results with best accuracy for $J=1$ and $J=2$.

We also observed for small changes in initial vector, for example taking $[0.8,0.8,\hdots,0.8]$ or $[0.7,0.7,\hdots,0.7]$ doesn't significantly change the solution.

\subsection{Example 3 (Thermal Explosion)}\label{P2_sec6c}
Consider the non linear SBVP:
\begin{eqnarray}\label{P2_63}
Ly(t)+e^{y(t)}=0,\quad y'(0)=0,\quad y(1)=0, ~k_g=1.
\end{eqnarray}
Above nonlinear SBVP is derived by Chamber \cite{CHAMBRE1952}. This equation arises in the thermal explosion in cylindrical vessel. The exact solution of this equation is $y(x) = 2 \ln{\frac{4-2\sqrt{2}}{(3-2\sqrt{2})x^2+1}}$.

Comparison Graphs taking initial vector [$0,0,\hdots, 0$] and $J=1$, $J=2$ are plotted in Figure \ref{exfig8}. Tables for solution and error are tabulated in table \ref{extab8} and table \ref{exerr2}.
\begin{table}[H]
\caption{\small{Comparison of HWNA, HeWNA, HWQA, HeWQA methods solution with analytical solution for example \ref{P2_sec6c} taking $J=2$:}}\label{extab8}			 \centering											
\begin{center}											
\resizebox{12cm}{2cm}{											
\begin{tabular}	{|c | l|  l|  l| l| l| }									
\hline											
              																					
Grid Points	&	   HWNA \cite{akvdt2018}	&	  HeWNA	&	 HWQA \cite{akvdt2018}	&	HeWQA	&	Exact	\\  \hline
0	&	0.316727578	&	0.316694368	&	0.316727578	&	0.316694368	&	0.316694368	 \\
  1/16	&	0.315388914	&	0.315354403	&	0.315388914	&	0.315354403	&	 0.315354404	\\
  3/16	&	0.304700946	&	0.304666887	&	0.304700946	&	0.304666887	&	 0.304666888	\\
  5/16	&	0.283494667	&	0.283461679	&	0.283494667	&	0.283461679	&	 0.283461679	\\
  7/16	&	0.252100547	&	0.252069555	&	0.252100547	&	0.252069555	&	 0.252069555	\\
  9/16	&	0.210993138	&	0.210965461	&	0.210993138	&	0.210965461	&	 0.210965462	\\
 11/16	&	0.160768168	&	0.16074555	&	0.160768168	&	0.16074555	&	 0.16074555	\\
 13/16	&	0.102115684	&	0.102100258	&	0.102115684	&	0.102100258	&	 0.102100258	\\
 15/16	&	0.035791587	&	0.035785793	&	0.035791587	&	0.035785793	&	 0.035785793	
\\\hline
\end{tabular}}											
\end{center}											
\label{Table2}											
\end{table}
\begin{table}[H]
	\caption{\small{Comparison of error of HWNA, HeWNA, HWQA, HeWQA methods for example \ref{P2_sec6c} taking $J=2$ :}}\label{exerr2}											 \centering											
	\begin{center}											
		\resizebox{8cm}{0.8cm}{											
			\begin{tabular}	{|c | l|  l|  l| l| }										
				\hline      																Error	&	  HWNA \cite{akvdt2018}	&	HeWNA	&	  HWQA \cite{akvdt2018}	&	HeWQA\\ \hline
				$L_\infty$	&	0.0000345103	& 1.07541$\times 10^{-10}$  	&	0.0000345103	&	1.07541$\times 10^{-10}$\\
				$L_2$	&	0.0000771278	&	4.99369$\times 10^{-10}$	&  0.0000771278	&		4.99369$\times 10^{-10}$ \\\hline
		\end{tabular}}											
	\end{center}
\end{table}

\begin{figure}[H]
\begin{center}
\includegraphics[scale=0.30]{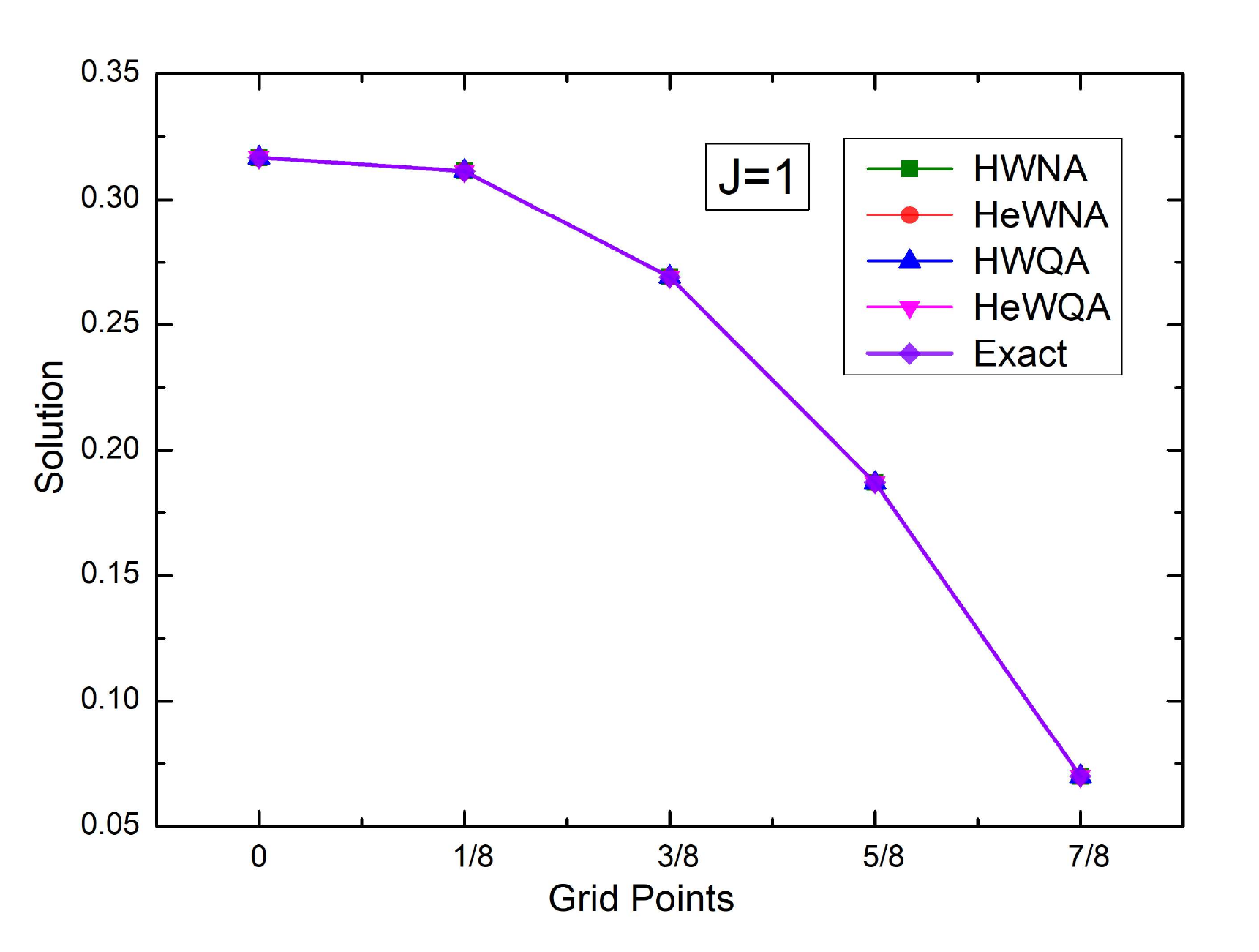}\includegraphics[scale=0.21]{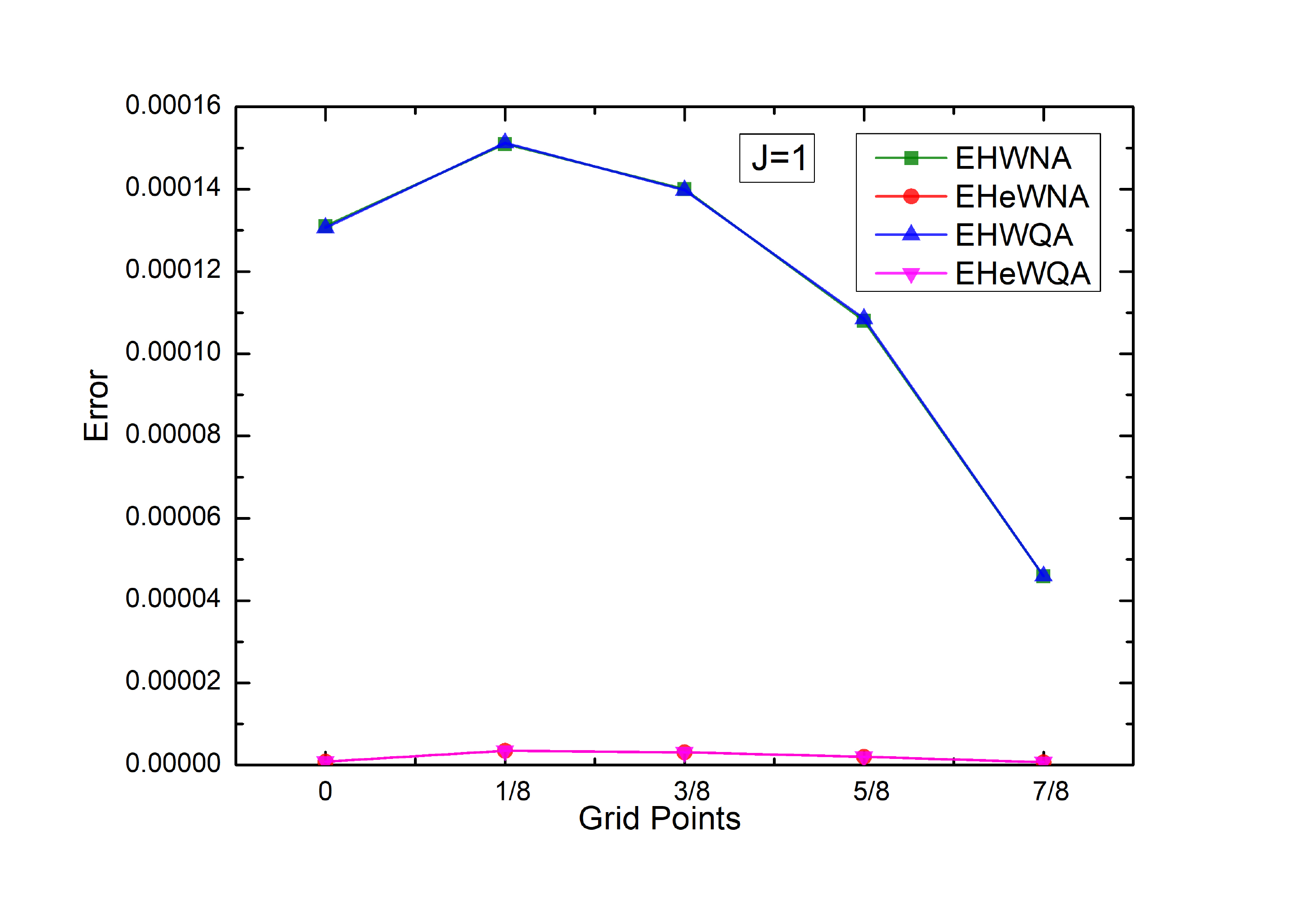}
\end{center}
\begin{center}
\includegraphics[scale=0.30]{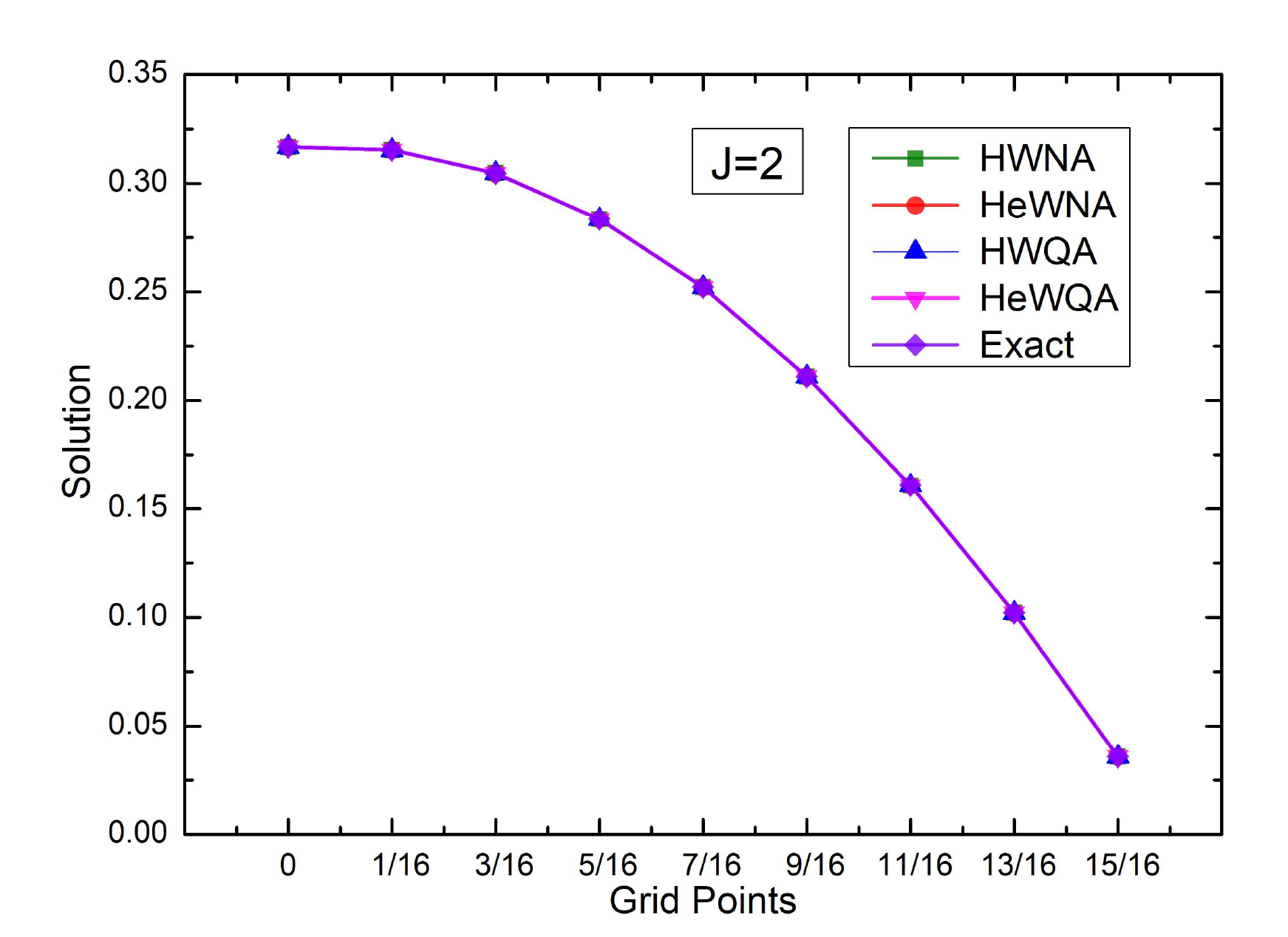}\includegraphics[scale=0.30]{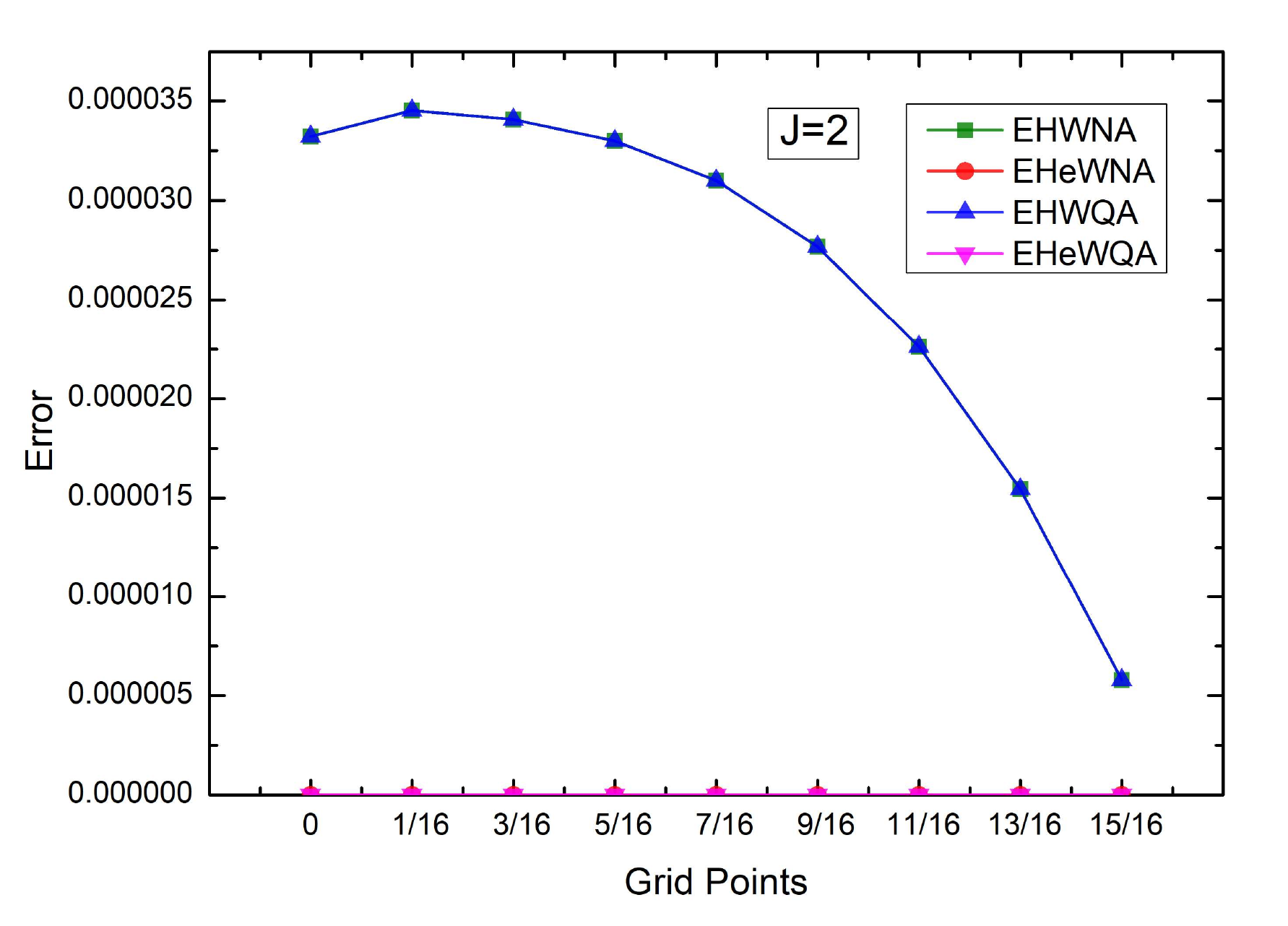}
\end{center}
\caption{Comparison plot and error plots of solution methods for $J=1,2$ for example \ref{P2_sec6c}}\label{exfig8}
\end{figure}
This is test case derived by Chambre  \cite{CHAMBRE1952} long back again exact solution is available. Table \ref{extab8} and figure \ref{exfig8} show that numerics are in good agreement with exact solutions or $J=1$ and $J=2$.

We also observed for small changes in initial vector, for example taking $[0.1,0.1,\hdots,0.1]$ or $[0.2,0.2,\hdots,0.2]$ doesn't significantly change the solution.
\subsection{Example 4 (Rotationally Symmetric Shallow Membrane Caps)}\label{P2_sec6d}
Consider the non linear SBVP:
\begin{eqnarray}\label{P2_65}
Ly(t)+\left(\frac{1}{8y^{2}(t)}-\frac{1}{2}\right)=0,\quad y'(0)=0,\quad y(1)=1,~k_g=1.
\end{eqnarray}
Above nonlinear SBVP is studied in papers \cite{RW1989,JV1998}. Exact solution of this problem is not known.

Comparison Graphs taking initial vector [$1,1,\hdots, 1$] and $J=1$, $J=2$ are plotted in Figure \ref{exfig9}. Tables for solution is tabulated in table \ref{extab9}.
\begin{table}[H]
\caption{\small{ Comparison of HWQA, HeWNA, HWQA, HeWQA method solution for example \ref{P2_sec6d} taking $J=2$:}}	\label{extab9}													 \centering	
\begin{center}	
\resizebox{10cm}{2cm}{\begin{tabular}{|c |l| l| l| l|}
\hline
Grid Points	&	HWNA \cite{akvdt2018}	&	HeWNA	&	HWQA \cite{akvdt2018}	&	HeWQA	\\	\hline
   0	&	0.954137376	&	0.954135008	&	0.954137376	&	0.954135008	\\
  1/16	&	0.954314498	&	0.954311604	&	0.954314498	&	0.954311604	\\
  3/16	&	0.95573187	&	0.95572956	&	0.95573187	&	0.95572956	\\
  5/16	&	0.958569785	&	0.958567713	&	0.958569785	&	0.958567713	\\
  7/16	&	0.962834546	&	0.962832683	&	0.962834546	&	0.962832683	\\
  9/16	&	0.968535496	&	0.968533886	&	0.968535496	&	0.968533886	\\
 11/16	&	0.975684891	&	0.975683641	&	0.975684891	&	0.975683641	\\
 13/16	&	0.984297738	&	0.984296771	&	0.984297738	&	0.984296771	\\
 15/16	&	0.994391588	&	0.994391728	&	0.994391588	&	0.994391728	
\\\hline
\end{tabular}}											
\end{center}
\end{table}
\begin{figure}[H]
\begin{center}
\includegraphics[scale=0.30]{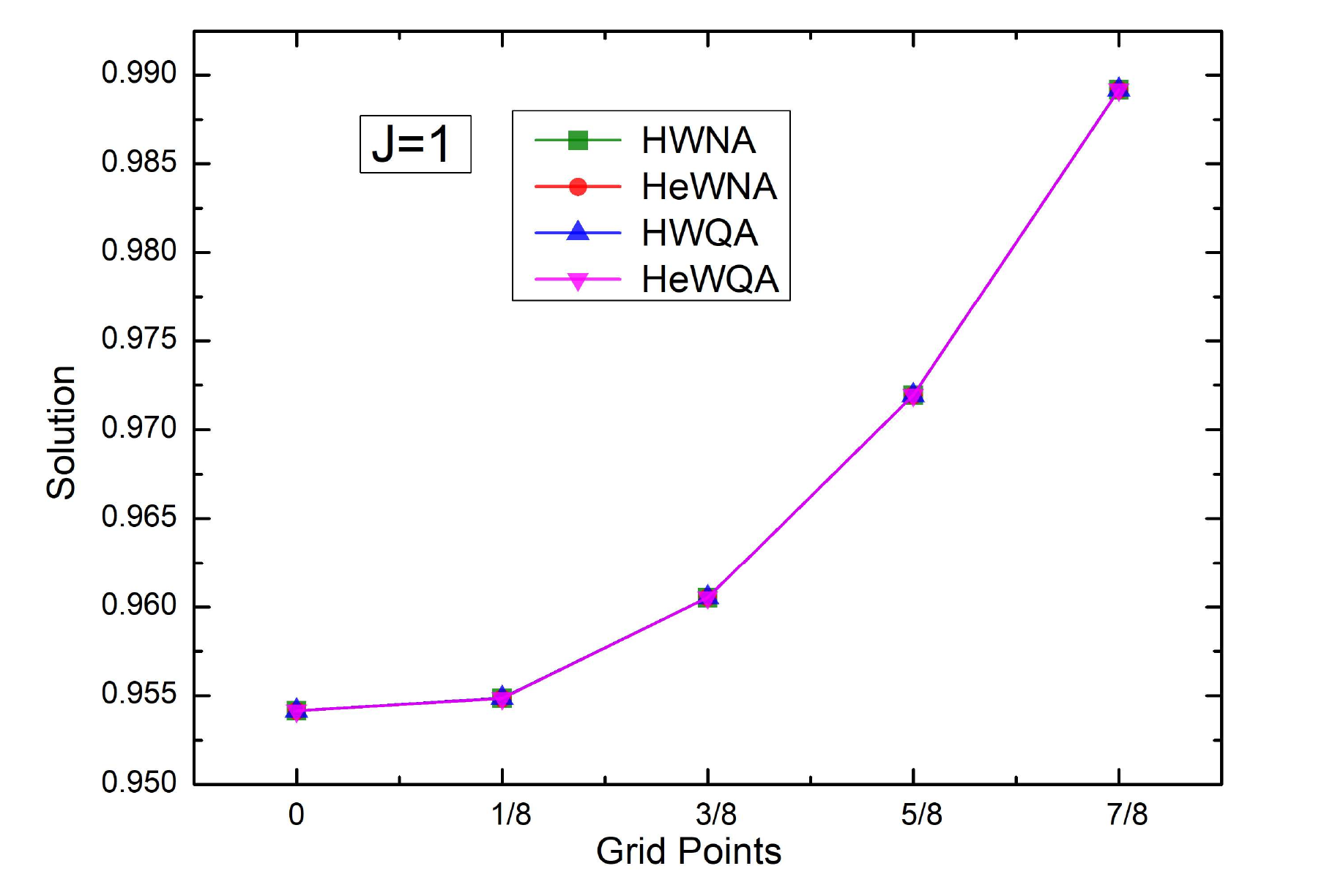}\includegraphics[scale=0.30]{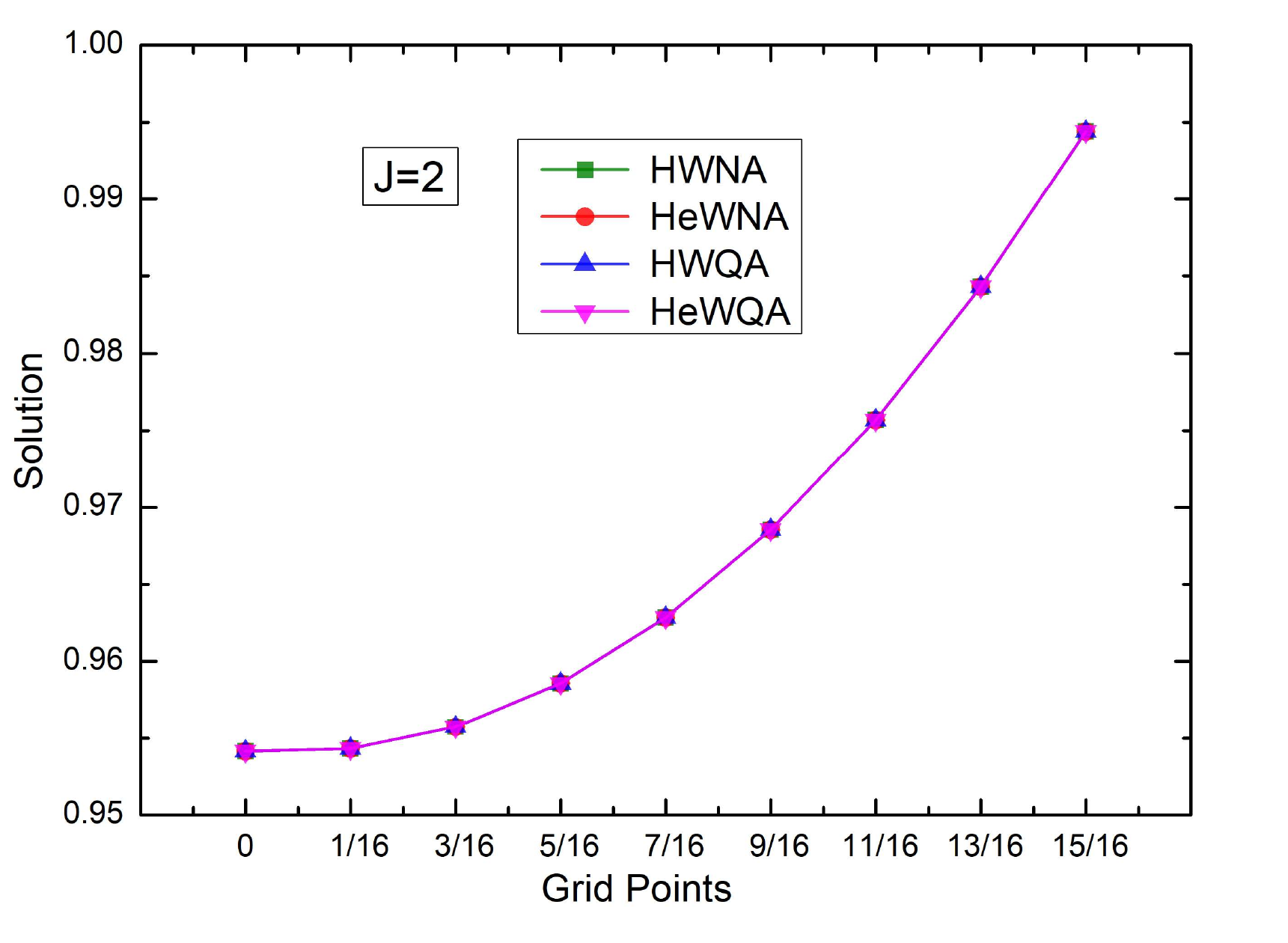}
\end{center}
\caption{Comparison plots of solution methods for $J=1,2$ for example \ref{P2_sec6d}}\label{exfig9}
\end{figure}
In this real life example again exact solution is not known so comparison is not done with exact solution. Table \ref{extab9} and figure \ref{exfig9} show that computed results are comparable for $J=1,2$.

We also observed for small changes in initial vector, for example taking $[0.9,0.9,\hdots,0.9]$ or $[0.8,0.8,\hdots,0.8]$ doesn't significantly change the solution.

\subsection{Example 5 (Thermal Distribution in Human Head)}\label{P2_sec6e}
Consider the non linear SBVP:
\begin{eqnarray}\label{P2_66}
Ly(t)+e^{-y(t)}=0,\quad y'(0)=0,\quad 2y(1)+y'(1)=0,~k_g=2.
\end{eqnarray}
This SBVP is derived by Duggan and Goodman \cite{RA1986}. Exact solution of this problem is not known to the best of our knowledge.

Comparison Graphs taking initial vector [$0,0,\hdots, 0$] and $J=1$, $J=2$ are plotted in Figure \ref{exfig10}. Tables for solution is tabulated in table \ref{extab10}.
\begin{table}[H]
\caption{\small{Comparison of HWQA, HeWNA, HWQA, HeWQA method solution for example \ref{P2_sec6e} taking $J=2$:}}\label{extab10}														 \centering											
\begin{center}	
\resizebox{10cm}{2cm}{
\begin{tabular}	{|c | l| l| l| l| }
\hline		
Grid Points	&	  HWNA \cite{akvdt2018}	&	  HeWNA	&	  HWQA \cite{akvdt2018}	&	HeWQA	\\	\hline
0	&	0.269855704	&	0.269948774	&	0.272263769	&	0.272366612	\\
  1/16	&	0.269358573	&	0.269451863	&	0.27176762	&	0.271870738	\\
  3/16	&	0.265377954	&	0.265471233	&	0.267793921	&	0.267896983	\\
  5/16	&	0.257388082	&	0.257481347	&	0.259810468	&	0.259913411	\\
  7/16	&	0.245331028	&	0.245424295	&	0.247745058	&	0.247847809	\\
  9/16	&	0.229118226	&	0.229211536	&	0.231489202	&	0.231591678	\\
 11/16	&	0.208628362	&	0.2087218	&	0.210897975	&	0.211000089	\\
 13/16	&	0.183704413	&	0.183798121	&	0.18579005	&	0.18589165	\\
 15/16	&	0.154149664	&	0.154243862	&	0.155947881	&	0.156048741	\\
1	&	0.13756259	&	0.137656718	&	0.139174003	&	0.139274111	
\\\hline
\end{tabular}}											
\end{center}											
\end{table}	
\begin{figure}[H]
\begin{center}
\includegraphics[scale=0.30]{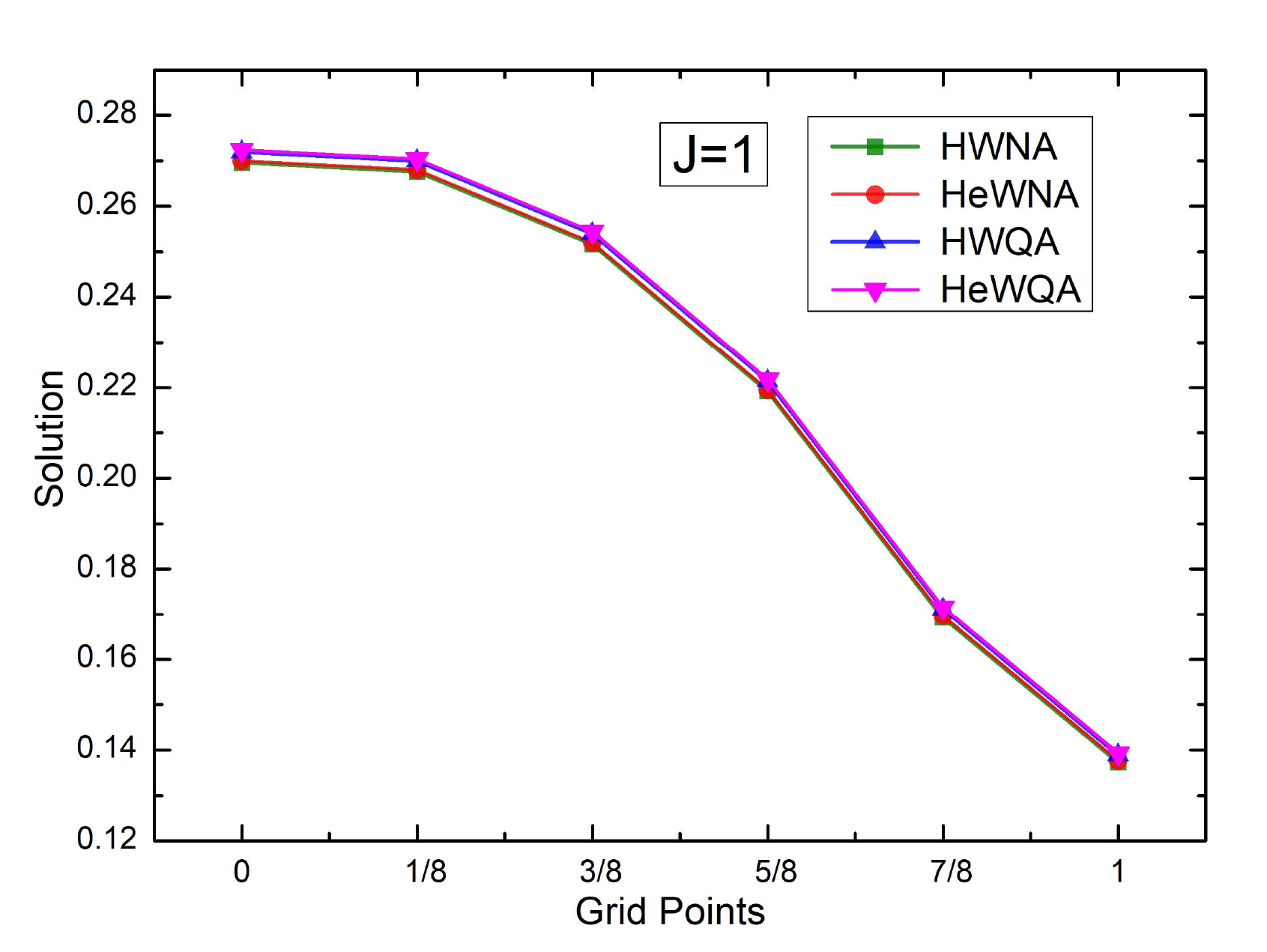}\includegraphics[scale=0.30]{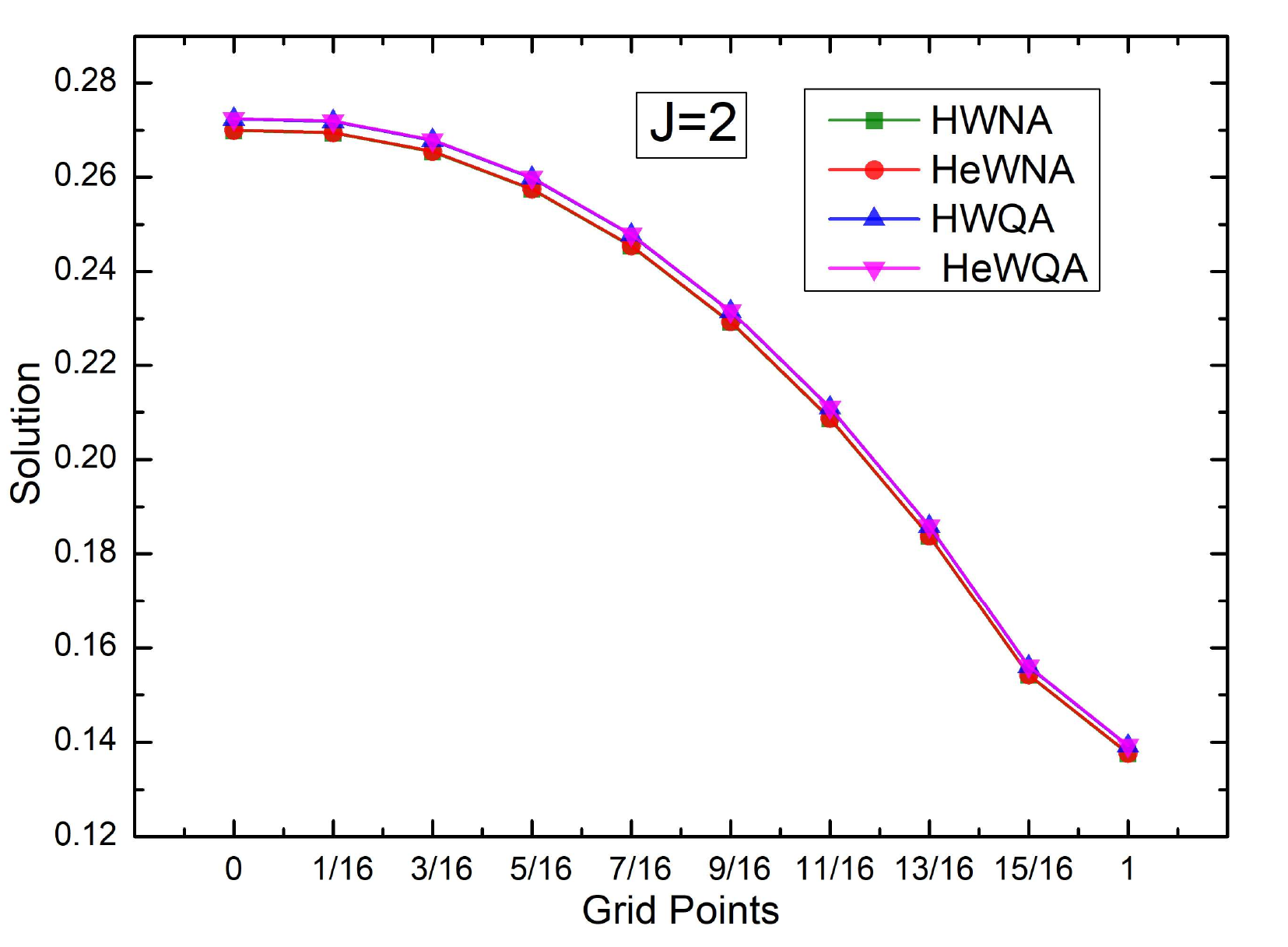}
\end{center}
\caption{Comparison plots of solution methods for $J=1,2$ for example \ref{P2_sec6e}}\label{exfig10}
\end{figure}
In absence of exact solution the comparison has not been made with exact solution. But comparison of all four methods for in the given problem due to Duggan and Goodman \cite{RA1986}, in table \ref{extab10} and figure \ref{exfig10} shows accuracy of the present method.

We also observed for small changes in initial vector, for example taking $[0.1,0.1,\hdots,0.1]$ or $[0.2,0.2,\hdots,0.2]$ doesn't significantly change the solution in any case.

\section{Conclusions}\label{P2_conclusions} In this research article, we have proposed a new model governing exothermic reactions and four different numerical methods based on wavelets, namely HWQA, HWNA, HeWQA, HeWNA for solving these nonlinear SBVPs arising in different branches of science and engineering (cf. \cite{RW1989, JV1998, RA1986, CHAMBRE1952, CS1967}). We have applied these methods in five real life examples [see equations (\ref{P2_61}), (\ref{P2_62}), (\ref{P2_63}), (\ref{P2_65}) and (\ref{P2_66})]. Singularity of differential equations can also be very well handled with help of these four proposed methods based. Difficulty arise due to non-linearity of differential equations is dealt with the help of quasilinearization in HWQA and HeWQA method. In the other two proposed method, HWNA and HeWNA, we will solve the resulting non-linear system with help of Newton-Raphson method. Boundary conditions are also handled well by the proposed methods. Main advantage of proposed methods is that solutions with high accuracy are obtained using a few iterations. We also observe that small perturbation in initial vector does not significantly change the solution. Which shows that our method is numerically stable.

Our convergence analysis shows that that $||E_{k,M}||$ tends to zero as $M$ tends to infinite. Which shows that accuracy of solution increases as $J$ increases.

Computational work illustrate the validity and accuracy of the procedure. Our computations are based on a higher resolution and codes developed can easily be used for even further resolutions.
\bibliography{MasterR}
\end{document}